\documentclass[12pt]{amsart}

\usepackage[utf8]{inputenc}

\usepackage[authoryear,round]{natbib}

\usepackage{orcidlink}

\usepackage{amsmath, amssymb, amsfonts, booktabs,hyperref}
\usepackage{tikz}
\usetikzlibrary{decorations.pathreplacing}
\usetikzlibrary{intersections,calc,arrows.meta,patterns}
\usepackage[normalem]{ulem}

\newcommand{\R}{\mathbb{R}}
\renewcommand{\P}{\mathbb{P}}
\newcommand{\N}{\mathbb{N}}
\newcommand{\Q}{\mathrm{Q}_{\Mast}}

\newcommand{\E}{\mathbb{E}}
\newcommand{\I}{\mathrm{I}}

\newcommand{\cH}{\eta}

\newcommand{\norm}{\,\mathrm{pen}}
\newcommand{\normd}{\,\mathrm{pen}^\ast}
\newcommand{\norms}{\,\mathrm{pen}_{\mathcal{M}^\ast}}
\newcommand{\normsd}{\,\mathrm{pen}_{\mathcal{M}^\ast}^\ast}
\newcommand{\poly}[1]{\left\|#1\right\|_\diamond}
\newcommand{\polyd}[1]{\left\|#1\right\|^\ast_\diamond}
\newcommand{\polydr}[2]{\left\|#1\right\|^\ast_{\diamond\mid #2}}
\newcommand{\proj}{\mathrm{\mathbf{P}}_{\mathcal{M}^\ast}}
\newcommand{\PM}{\mathrm{P}_{\mathcal{M}^\ast}}

\newcommand{\PP}{\mathrm{P}}
\newcommand{\Sast}{\mathcal{S}^\ast}
\newcommand{\Mast}{\mathcal{M}^\ast}
\renewcommand{\SS}{\mathcal{S}}

\newcommand{\cnorm}[1]{[\mkern-1.5mu[#1]\mkern-1.5mu]}

\DeclareMathOperator*{\argmin}{\arg\min}
\DeclareMathOperator*{\argmax}{\arg\max}
\newcommand{\tr}[1]{\mathrm{tr}\left(#1\right)}
\renewcommand{\v}[1]{\mathrm{vec}(#1)}
\newcommand{\vp}[1]{\mathrm{vec_+}(#1)}
\newcommand{\patt}[1]{\mathrm{patt}_{\diamond}\!\left(#1\right)}
\newcommand{\scalar}[2]{\left\langle\,#1\,\middle|\,#2\,\right\rangle}

\makeatletter
\DeclareRobustCommand{\opnorm}{\@ifstar\@opnorms\@opnorm}

\newcommand{\@opnorms}[1]{%
  \ensuremath{%
    \left|\mkern-1.5mu\left|\mkern-1.5mu\left|
    #1
    \right|\mkern-1.5mu\right|\mkern-1.5mu\right|
  }%
}

\newcommand{\@opnorm}[2][]{%
  \ensuremath{%
    \mathopen{#1|\mkern-1.5mu#1|\mkern-1.5mu#1|}
    #2
    \mathclose{#1|\mkern-1.5mu#1|\mkern-1.5mu#1|}
  }%
}
\makeatother

\makeatletter
\DeclareRobustCommand{\copnorm}{\@ifstar\@copnorms\@copnorm}

\newcommand{\@copnorms}[1]{%
  \ensuremath{%
    \left[\mkern-1.5mu\left[\mkern-1.5mu\left[
    #1
    \right]\mkern-1.5mu\right]\mkern-1.5mu\right]
  }%
}

\newcommand{\@copnorm}[2][]{%
  \ensuremath{%
    \mathopen{#1[\mkern-1.5mu#1[\mkern-1.5mu#1[}
    #2
    \mathclose{#1]\mkern-1.5mu#1]\mkern-1.5mu#1]}
  }%
}
\makeatother

\usepackage{color,graphicx}

\usepackage[margin=30mm]{geometry}

\newtheorem{theorem}{Theorem}[section]
\newtheorem*{theorem*}{Theorem}
\newtheorem{lemma}[theorem]{Lemma}
\newtheorem{corollary}[theorem]{Corollary}

\theoremstyle{definition}

\newtheorem{definition}[theorem]{Definition}
\newtheorem{remark}[theorem]{Remark}
\newtheorem{example}[theorem]{Example}

\title[From Graphical Lasso to Atomic Norms]{From Graphical Lasso to Atomic Norms: High-Dimensional Pattern Recovery}

\author[P. Graczyk]{Piotr Graczyk}
\email{piotr.graczyk@univ-angers.fr}
\address{Univ Angers, CNRS, LAREMA, SFR MATHSTIC, F-49000 Angers, France}
\address{Faculty of Pure and Applied Mathematics, Wrocław University of Science and Technology, Wybrzeże Wyspiańskiego 27, 50-370 Wroc\l{}aw, Poland}
\author[B. Kołodziejek]{Bartosz Kołodziejek \orcidlink{0000-0002-5220-9012}}
\email{bartosz.kolodziejek@pw.edu.pl}
\address{Faculty of Mathematics and Information Sciences, Warsaw University of Technology, Koszykowa 75, \mbox{00-662} Warsaw, Poland}
\author[H. Nakashima]{Hideto Nakashima}
\email{h-nakashima@tokai.ac.jp}
\address{Department of Mathematics, Faculty of Science, Tokai University, 4-1-1, Kitakaname, Hiratsuka, Kanagawa, 259-1292, Japan}

\author[M. Wilczyński]{Maciej Wilczyński}
\email{maciej.wilczynski@pwr.edu.pl}
\address{Faculty of Pure and Applied Mathematics, Wrocław University of Science and Technology, Wybrzeże Wyspiańskiego 27, 50-370 Wroc\l{}aw, Poland}

\thanks{For the purpose of Open Access, the authors have applied a CC-BY public copyright licence to any Author Accepted Manuscript (AAM) version arising from this submission.
}
\begin{document}

\begin{abstract}
We study high-dimensional precision-matrix estimation with polyhedral atomic-norm penalties on the off-diagonal entries. This framework covers the graphical lasso, $\ell_\infty$-type penalties, SLOPE/OWL penalties, and penalties motivated by colored graphical models. We define the target pattern through the subdifferential face of the dual ball and derive sufficient conditions for exact pattern recovery from the sample covariance matrix. Our analysis introduces two geometric thresholds, a generalized irrepresentability condition, and extends the primal-dual witness method beyond the $\ell_1$ setting. For the graphical lasso, our results yield weaker covariance-deviation requirements and a less restrictive irrepresentability condition than existing theory, leading to
sharper asymptotic guarantees. Numerical experiments for several graph topologies show that the resulting bounds are often close to the empirical recovery threshold.
 
\end{abstract}


\maketitle

\section{Introduction}

Estimating a high-dimensional precision matrix is a central problem in modern statistics. In Gaussian graphical models, the precision matrix encodes conditional-independence relations, and its zero pattern determines the underlying graph. A widely used approach is the graphical lasso, which combines the Gaussian log-likelihood with an $\ell_1$ penalty to promote sparsity in the off-diagonal entries of the precision matrix \citep{YuanLi07,glasso3,Friedman2008}. Foundational results by \cite{Ravi2008,ravikumar2011graphical} established conditions under which this estimator consistently recovers the support and sign pattern of the true precision matrix; see also \citep{Rothman08,Sparsistency}.

Many models of practical interest, however, exhibit structure that is richer than simple sparsity. In particular, colored graphical models may involve not only conditional-independence zeros but also symmetry and equality constraints among entries of the precision matrix \citep{HL08}. Such patterns are not naturally captured by the $\ell_1$ norm alone. Recovering them requires penalties that can encode clustering, equality, or more general geometric structure in addition to sparsity.


We develop a unified pattern-recovery theory for precision-matrix estimators penalized by polyhedral atomic norms, covering $\ell_1$, $\ell_\infty$, SLOPE/OWL \citep{ellinfty,bogdan2015slope,figueiredo16,OSCAR08,graphSLOPE}, and penalties relevant to colored graphical models \cite{GaoMassam15,LSWG21}. The broader problem of model identification under structured convex regularization was studied in \cite{Vaiter15,Vaiter18} through model subspaces, partial smoothness, and generalized irrepresentability conditions. In the polyhedral setting, we use the more explicit notion of a pattern developed by \citet{POLYNORMS}: two parameters have the same pattern when their subdifferentials coincide, and the resulting pattern classes correspond to relative interiors of normal cones of faces of the dual polytope.   This perspective allows us to study, within a common framework, support patterns, clustered coefficients, and equality-constrained structures.

A natural language for constructing such penalties is provided by atomic norms \citep{Chandra12}. The basic idea is that a structured parameter can often be represented as a combination of a small number of elementary atoms drawn from a prescribed set. The convex hull of these atoms then defines a penalty whose geometry reflects the target notion of simplicity. In this paper, we restrict attention to atomic norms whose unit balls are polytopes, for which the face-based description of patterns is finite and explicit. 

Our main theorem gives deterministic conditions for exact pattern
recovery in penalized precision-matrix estimation. The conditions use two geometric stability thresholds and an irrepresentability condition.
The proof controls the nonlinear remainder in the matrix-inverse
expansion by a primal-dual witness argument.

When specialized to the $\ell_1$ penalty, our results sharpen the classical GLASSO theory. In particular, we obtain weaker requirements on the deviation of the sample covariance matrix from its population counterpart and a less restrictive irrepresentability condition than in \citet{ravikumar2011graphical}. For other penalties, including $\ell_\infty$ and SLOPE/OWL, our results provide new pattern-recovery guarantees.

\subsection{Problem Setup}\label{sec:setup}

Let $X=(X_1,\ldots,X_p)^\top$ be a zero-mean $p$-dimensional random vector with finite second moments. Define the true covariance matrix and precision matrix by
\[
\Sigma^\ast=\E[XX^\top],
\qquad
K^\ast=(\Sigma^\ast)^{-1},
\]
and assume throughout that $\Sigma^\ast$ is positive definite.

Let $(X^{(i)})_{i=1}^n$ be i.i.d.\ copies of $X$, and let
\[
\hat\Sigma=\frac1n\sum_{i=1}^n X^{(i)}(X^{(i)})^\top
\]
denote the sample covariance matrix. We consider estimators of the form
\[
\hat K
=
\argmin_{K\in\mathrm{Sym}_+(p)}
\left\{
\tr{\hat\Sigma K}-\log\det(K)+\lambda \norm(K)
\right\},
\]
where $\lambda\ge 0$ is a tuning parameter and $\norm(\cdot)$ is a penalty acting on the off-diagonal entries of $K$.
More precisely, we assume that
\begin{align}\label{eq:penK}
\norm(K)=2\,\poly{\vp{K}},
\end{align}
where $\vp{K}\in\R^{p(p-1)/2}$ denotes the vector of strictly lower-triangular entries of $K$, and $\poly{\cdot}$ is a polyhedral atomic norm on $\R^{p(p-1)/2}$. Thus, diagonal entries are left unpenalized. The factor of $2$ is included so that, in the special case $\poly{x}=\|x\|_1$, one recovers the usual off-diagonal $\ell_1$ penalty $
\norm(K)=\sum_{i\neq j}|K_{ij}|$,
corresponding to the graphical lasso \citep{YuanLi07,glasso3,Friedman2008}.

The loss function $\tr{\hat\Sigma K}-\log\det(K)$ 
is the Gaussian negative log-likelihood up to an additive constant. At the same time, the same objective can be interpreted more generally through the Bregman divergence associated with the log-determinant barrier. Specifically, if $f(K)=-\log\det(K)$, then
\[
D_f(A\mid\mid B)
=
-\log\det(A)+\log\det(B)+\tr{B^{-1}(A-B)}.
\]
For invertible $K$, one has
\[
\tr{\hat\Sigma K}-\log\det(K)
=
D_f(K\mid\mid K^\ast)+\tr{(\hat\Sigma-\Sigma^\ast)K}-\log\det(K^\ast)+p.
\]
This interpretation allows one to view the estimator as minimizing a penalized Bregman divergence, and hence extends naturally beyond the Gaussian setting under finite-second-moment assumptions \citep{ravikumar2011graphical,zwiernik2023entropic}.

The support of the estimator $\hat K$ defines an estimated partial-correlation graph. Under Gaussianity, zeros in $K^\ast$ are equivalent to conditional independences. More generally, they encode vanishing partial correlations, and under elliptical or transelliptical models this interpretation remains particularly meaningful \citep{RosselZwiernik21}.

Our goal is to determine conditions under which $\hat K$ recovers the same {pattern} as $K^\ast$, where pattern is understood in the geometric sense induced by the penalty $\poly{\cdot}$. For the $\ell_1$ norm this reduces to sign and support recovery, whereas for other polyhedral atomic norms it may encode clustering, equality constraints, or more general structured configurations.

Despite their different interpretations, all the covered examples enter our analysis through the same geometric objects: a face of the dual polytope, its associated pattern subspace, and the corresponding stability
thresholds. Section~\ref{sec:preli} introduces this common geometric description and explains its role in the pattern-recovery analysis.

\subsection{Examples of patterns and scope of the framework}\label{sec:12}

The following examples describe the patterns induced by the penalties
considered in this paper. The first three penalties are also used in
our theoretical comparisons and numerical experiments.\\
\noindent \textbf{$\ell_1$ norm}:
The induced pattern records the support and signs of the coordinates. For precision matrices, this gives the usual signed-support pattern of the off-diagonal entries and leads to the graphical lasso specialization studied in Section~\ref{sec:glasso}
\citep{YuanLi07,glasso3,Friedman2008,ravikumar2011graphical}.\\
\noindent\textbf{$\ell_\infty$ norm}:
The induced pattern records which coordinates attain the largest absolute value, together with their signs \citep{ellinfty}. Thus, it describes a max-type clustering structure rather than ordinary sparsity.\\
\noindent\textbf{SLOPE/OWL norms}:
These norms take the form
\[
    \poly{x}=\sum_{j=1}^m w_j |x|_{(j)},
\]
where \(w_1\geq\cdots\geq w_m\geq0\) and \(w_1>0\). Their patterns encode zeros, signs, clusters of equal absolute values, and the ordering of these clusters
\citep{bogdan2015slope,figueiredo16,OSCAR08,schneider2020geometry,7aut}.
Graphical versions of SLOPE were studied by \citet{Mazza20,graphSLOPE}.

The framework also accommodates other polyhedral penalties.\\
\noindent\textbf{ Generalized LASSO}:
If \(\ker(A)=\{0\}\), then $ \poly{x}=\|Ax\|_1$
is a polyhedral norm and is covered by our framework \citep{Tib11}. The kernel condition ensures that the penalty is a norm rather than a seminorm. This class includes the fused LASSO \citep{tibshirani2005sparsity} and graph-guided fused penalties \citep{GFLASSO,fusedglasso,fusedMglasso} when their defining operators are augmented, for example by an \(\ell_1\) term, so that the resulting operator is injective. The pure difference penalties do not satisfy this condition.\\
\noindent\textbf{GOLAZO}:
The GOLAZO penalty \citep{GOLAZO} is given by
\[
    \norm(K)
    =
    \sum_{i\neq j}
    \max\{L_{ij}K_{ij},U_{ij}K_{ij}\},
\]
where \(L\) and \(U\) are symmetric and \(L_{ij}<0<U_{ij}\) for every \(i\neq j\). Under these strict inequalities, this is a generally asymmetric polyhedral gauge of the off-diagonal entries with a compact unit ball and is therefore covered by our framework.

The generalized-LASSO formulation also clarifies the scope of our framework for colored graphical models. Such models may impose equality constraints on both diagonal entries, corresponding to vertex colors, and off-diagonal entries, corresponding to edge colors \citep{HL08}. Fusion penalties proposed in this setting combine these equality structures with sparsity. A representative form is
\[
    \norm(K)
    =
    \lambda_1 \sum_{i<i'} |K_{ii}-K_{i'i'}|
    + \lambda_2 \sum_{i<j}|K_{ij}|
    + \lambda_3
      \sum_{\{(i,j),(i',j')\}\in\mathcal C_E}
      |K_{ij}-K_{i'j'}|,
\]
where \(\lambda_1,\lambda_3\geq0\), \(\lambda_2>0\), and \(\mathcal C_E\) is a specified collection of pairs of off-diagonal entries \citep{GaoMassam15,LSWG21}. The three terms promote, respectively, equality among diagonal entries, sparsity of the off-diagonal entries, and equality among selected off-diagonal entries.

Our results cover this penalty when $\lambda_1=0$ and $\lambda_2>0$.
The sparsity term $\lambda_2$ makes the defining operator injective, so the
remaining off-diagonal penalty is a polyhedral norm. Thus, we cover
edge-color constraints, but not diagonal vertex-color constraints.
 Treating such constraints would require a separate penalty on the diagonal entries and a corresponding modification of the theoretical analysis.

These examples also illustrate the extent to which a pattern admits a graphical interpretation. In the Gaussian setting, the interpretation of the \(\ell_1\)-pattern is direct: zero off-diagonal entries of the precision matrix correspond to missing edges in the conditional-independence graph, while the pattern additionally records the signs of the nonzero entries \citep{ravikumar2011graphical}. For colored Gaussian graphical models, equality relations among off-diagonal entries can be interpreted as edge-color classes \citep{HL08,GaoMassam15,LSWG21}. Clusters of equal absolute values induced by SLOPE/OWL penalties may similarly group edge parameters, while also recording their signs and the ordering of the clusters \citep{schneider2020geometry,7aut}. For an arbitrary polyhedral penalty, however, a pattern need not have a conditional-independence or other intrinsic graphical interpretation: it may instead encode clustering, ordering, maximal coordinates, or other structural relations. The notion of a pattern used here should therefore be understood as a geometric generalization of graphical structure rather than as a graphical object in every case.

See also \citet{Chandra12} for a broader overview of atomic norms and their applications.

\subsection{Contributions of the paper}

This paper extends high-dimensional pattern-recovery theory for precision-matrix estimation from the $\ell_1$ setting to a broad class of polyhedral atomic norms. The definition of a pattern through equality of subdifferentials, its facial and normal-cone representation, and the associated pattern subspaces are inherited from \citet{POLYNORMS}. 
We use this geometry to derive quantitative recovery bounds for
penalized precision-matrix estimators.

\emph{General pattern-recovery theory.}
We establish deterministic sufficient conditions under which the estimator \eqref{eq:hatK} recovers the true pattern induced by an arbitrary polyhedral atomic norm. The result is formulated in terms of a generalized irrepresentability condition and a deviation bound on $\hat\Sigma-\Sigma^\ast$. This substantially generalizes the support-recovery theory developed for the graphical lasso by \citet{Ravi2008,ravikumar2011graphical}.

\emph{Two geometric thresholds.}
We define explicit face-stability and pattern-stability thresholds for
the precision-matrix problem.

 The {face-stability threshold} $\tau_\diamond$ (Definition \ref{def:tau}) controls how far one may perturb the face projection in directions orthogonal to the pattern subspace while remaining in the face.
The {pattern-stability threshold} $\zeta_\diamond$ (Definition \ref{def:zeta}) quantifies the minimal perturbation within the pattern subspace required to change the pattern. Their roles are related to the subdifferential/nondegeneracy and local model-stability quantities appearing in \citet{Vaiter15,Vaiter18}, but the explicit thresholds below are adapted to our polyhedral precision-matrix analysis.

\emph{Primal-dual witness analysis for general polyhedral atomic norms.}
Our proofs build on the primal-dual witness method of \citet{ravikumar2011graphical}, but require substantial new ideas in the general atomic-norm setting.
A key technical ingredient is the function \eqref{eq:defF}, which is used to control the fixed point of the KKT conditions and leads to improved asymptotic bounds.

We also note that \citet{FUN_GLASSO} extended the primal-dual witness method to the functional graphical lasso setting.

\emph{Sharper results for the $\ell_1$ penalty.}
When specialized to the $\ell_1$ norm, our theory yields improved pattern-recovery guarantees for the graphical lasso. Compared with \citet{ravikumar2011graphical}, we obtain weaker requirements on $\|\hat\Sigma-\Sigma^\ast\|$ and a less restrictive irrepresentability condition. The improvement is not limited to constants: in some examples the classical bound is asymptotically more restrictive by a polynomial factor in $p$; see Example \ref{ex:comp}. Our comparison with \citet{ravikumar2011graphical} also reveals that one step in their proof 
 has a gap and requires an additional factor in the deviation bound.
We discuss this point and provide a corrected restatement of the result of \citet{ravikumar2011graphical} in Section~\ref{sec:glasso}.

\emph{Numerical validation.}
We investigate the resulting bounds for several graph topologies and several penalties, including $\ell_1$, $\ell_\infty$, and SLOPE. The experiments show that the theoretical thresholds can be quite sharp in practice, especially in the $\ell_1$ case, and they illustrate how the geometry of the penalty influences recoverability.

\subsection{Structure of the paper}

Section \ref{sec:preli} introduces the geometric preliminaries underlying our framework. We review atomic norms and their facial structure, define the notion of pattern via the subdifferential, and introduce the associated pattern subspaces. We then define the two main geometric quantities of the paper: the face-stability threshold $\tau_\diamond$ in Section \ref{sec:tau} and the pattern-stability threshold $\zeta_\diamond$ in Section \ref{sec:zeta}.

Section \ref{sec:main} contains our main theoretical results. We formulate general pattern-recovery theorems for precision-matrix estimators penalized by polyhedral atomic norms, and explain how they translate into finite-sample conditions through standard covariance-concentration arguments. We then specialize the theory to the $\ell_1$ penalty and compare it with the classical graphical lasso results of \citet{ravikumar2011graphical}.

Section \ref{sec:num} reports numerical experiments for several graph families and several penalties. We compare our admissible deviation thresholds with those implied by earlier theory and examine how closely the deterministic recovery threshold matches the empirical phase transition.

Section \ref{sec:ProofSkeleton} provides a roadmap of the proofs. We describe the restricted problem, the primal-dual witness construction, the key residual bounds, and the fixed-point argument that underlies the main theorems.

Section \ref{app:pr} treats pattern recovery in the case where the face-stability threshold fails to exist.

Section \ref{sec:SLOPE} discusses the choice of tuning parameters for SLOPE/OSCAR-type penalties in light of our theoretical results.

Additional technical material is collected in the Appendix. It contains detailed proofs for all results, which
were omitted from the main text for brevity and background on dual gauges.

\section{Preliminaries}\label{sec:preli}

 This section introduces the geometric objects used in the pattern-recovery analysis.
 We represent patterns by faces of the dual ball and associate each
pattern with a linear subspace. 
    We then introduce the restricted dual gauge and the matrix notation needed for the primal-dual witness argument.
    Finally, the two stability thresholds $\tau_\diamond$ and $\zeta_\diamond$ quantify, respectively, stability in off-pattern directions and preservation of the pattern within the pattern subspace;
    these quantities enter the recovery conditions in Section~\ref{sec:main}.

    Two roles of norms should be distinguished.
The penalty gauge determines the target pattern. The auxiliary norm
$\cnorm{\cdot}$ measures covariance perturbations and affects the
recovery bound and the prescribed tuning parameter. 
Its choice affects the bounds and the prescribed tuning
parameter, but not the definition of the target pattern.
    
\subsection{Atomic norms}\label{sec:atomic}

Let $\mathcal{A}\subset\R^m$ be a finite set whose points are called atoms,
and assume that $0\in\mathrm{int}(\mathrm{conv}(\mathcal{A}))$.
The gauge associated with $\mathcal{A}$, also called its Minkowski
functional, is defined by
\[
\|x\|_{\mathcal{A}}
=
\inf\{t>0\colon x\in t\,\mathrm{conv}(\mathcal{A})\}.
\]
If $\mathcal{A}$ is centrally symmetric about the origin
(i.e.\ $a\in\mathcal{A}\iff -a\in\mathcal{A}$), then
$\|\cdot\|_{\mathcal{A}}$ is a norm, called an atomic norm.
Following the convention in \cite{Chandra12}, we use the term
``atomic norm'' also for gauges that need not be symmetric.
The framework of \cite{Chandra12} allows general compact atom sets,
whereas we restrict attention to finite sets.

As the set $\mathcal{A}$ will be fixed, 
throughout the paper we will denote the atomic norm $\|\cdot\|_{\mathcal{A}}$ as $\poly{\cdot}$. Its unit ball is the polytope
\[
B=\{x\in\R^m\colon\poly{x}\leq1\}
=\mathrm{conv}(\mathcal A).
\] 
The corresponding dual gauge is defined by
\[
\polyd{y}:=\sup\{x^\top y\colon x\in\R^m,\ \poly{x}\leq1\},
\qquad y\in\R^m.
\]
Its unit ball
$B^\ast=\{y\in\R^m\colon\polyd{y}\leq1\}$ is the polar of $B$, and is again a polytope containing
the origin in its interior. Thus, the dual gauge is also
an atomic norm in the terminology adopted here.

Let $v_1,\ldots,v_K$ be the vertices of $B^\ast$.
Then
\begin{equation}\label{eq:atomicnorm}
\poly{x}=\max\{v_1^\top x,\ldots,v_K^\top x\}.
\end{equation}
We note that $\poly{\cdot}$ is a norm if for every $i$, there exists $j$ such that $v_j=-v_i$.
This representation connects the penalty to the faces
of its dual ball and will be used to describe patterns
in the next subsection, following \cite{POLYNORMS}.
 Note that in \cite{POLYNORMS}, such functions are generally called polyhedral gauges.

\begin{example}\label{ex:SLOPE}
 For $w_1\geq w_2\geq\ldots \geq  w_m\geq 0$ with $w_1>0$, define the SLOPE norm on $\R^m$ and its dual by
 \[
	\|x\|_w = \sum_{i=1}^{m} w_i |x|_{(i)}
	\quad\mbox{and}\quad 
	\| y\|_{w}^\ast = \max_{i=1,\ldots,m}\left\{\frac{ \sum_{k=1}^i |y|_{(k)}}{\sum_{k=1}^i w_k}\right\},
	\]
    where $|z|_{(k)}$ is the $k$-th largest absolute value of the coordinates of $z$, $k=1,\ldots,m$.

Clearly, $\|\cdot\|_w$ is an atomic norm. Special cases of the SLOPE norm include the $\ell_1$ and $\ell_\infty$-norms.
\end{example}

\subsection{Notion of pattern}\label{sec:pattern}

We now give a common definition of the structures described
in Section~\ref{sec:12}, so that recovery of signed support and
recovery of clusters can be treated within the same framework.

For a convex function $f\colon\R^m\to\R$, its subdifferential at $x$ is
\[
\partial f(x)=\{s\in\R^m\colon
f(y)\geq f(x)+s^\top(y-x)\text{ for all }y\in\R^m\}.
\]
We say that two vectors $x, y\in\R^m$ have the same pattern with respect to the atomic norm $\poly{\cdot}$ if
\[
\partial \poly{x} = \partial \poly{y}.   
\]
In this case, we write $\patt{x}=\patt{y}$. Let $C_x$ denote the pattern equivalence class of $x$. 
The pattern subspace generated by $x\in \R^m$ (referred to as the model subspace in \cite{Vaiter15,Vaiter18}) is defined as the linear span of $C_x$ and is denoted by $\SS_x$.  Membership in $\SS_x$ does not imply membership in $C_x$, and hence does not guarantee the same pattern as $x$.

By definition and the fact that subdifferentials are faces of $B^\ast$,
there is a one-to-one correspondence between patterns and faces of $B^\ast$. 
Let us see this correspondence in more detail.

It is well known (see, e.g., \cite[Exercise 8.31]{Rockafellar}) that if a function $f$ is given by $f=\max\{f_1,\ldots,f_K\}$,
where each $f_i$ is a smooth function, then the subdifferential of $f$ at a point $x$ is 
\[
\partial f(x) = \mathrm{conv}\{ \nabla f_i(x)\colon i\in\I_x\},
\]
where $\I_x = \{i\in\{1,\ldots,K\}\colon f_i(x)=f(x)\}$ is called the active set of indices. Applying this to the atomic norm
of the form \eqref{eq:atomicnorm}, we obtain
\[
 \partial\poly{x} = \mathrm{conv}\{v_i\colon i\in\I_x\}, 
\]
where the active set of indices is
\[
\I_x =\{ i\in\{1,\ldots,K\}\colon  v_i^\top x = \poly{x}\}. 
\]
We see that $\partial\poly{x}$ is a face of $B^\ast$, and write $F_x=\partial\poly{x}$.
In particular, we have 
\[
B^\ast=\partial\poly{0}= \mathrm{conv}\{v_1,\ldots,v_K\}. 
\]
Thus, for any $x\in \R^m$, 
its pattern equivalence class $C_x$ is in one-to-one correspondence with 
a face of $B^\ast$, and hence with its active set of indices $\I_x$.  
The set of all patterns is then $\mathcal{I} = \{ \I_x\colon x\in\R^m\}\subset 2^K$.
When $\I=\I_x$, we write 
\[
F_x=F_{\I},\quad C_x = C_{\I}\quad\mbox{and}\quad \SS_x = \SS_{\I}.
\]

By \cite[Theorem 3.2]{POLYNORMS}, $C_{\I}$ equals the relative interior of the normal cone $N_{F}$ of the face $F=F_{\I}$
of the dual unit ball $B^\ast$. 
It is known from convex analysis that the relative interiors of the normal cones of the faces form a partition ${\mathcal P}$ of $\R^m$ associated with the convex polytope $B^\ast$:
\[ 
\R^m=\bigcup_{F \mbox{ face of } B^\ast}\!\!\mathrm{ri}(N_{F}).  
\]

Assigning a pattern to $x\in \R^m$ corresponds to specifying which cone in the partition  ${\mathcal P}$ the vector $x$ belongs to.

For $\I\in\mathcal{I}$,  the pattern subspace is given by
\begin{align*}
\SS_{\I} = \mathrm{span}\{y\in\R^m\colon \I_y=\I\}= \{ y\in\R^m\colon v_i^\top y = v_j^\top y
\mbox{ for all }i,j\in \I\}.
\end{align*}
In other words, $\SS_{\I}$ is the orthogonal complement of the
direction space $\mathrm{span}\{v_i-v_j\colon i,j\in\I\}$
of the face $F_{\I}$.
If $|\I|= 1$, then $\SS_{\I} = \R^m$. 
Otherwise, 
$\SS_{\I} = \ker(H_{\I}^\top)$,
where, for a fixed $i_0\in \I$,  the matrix $H_{\I}$ of size $m\times (|\I|-1)$ is given by
\[
H_{\I} = (v_i-v_{i_0})_{i\in\I\setminus\{i_0\}}. 
\]
Thus, the orthogonal projection onto $\SS_{\I}$ is given by 
\[
\PP_{\I} =  I_m - H_{\I} (H_{\I}^\top H_{\I})^+ H_{\I}^\top,
\]
where $A^+$ denotes the Moore-Penrose inverse of $A$.
By definition, we have $\PP_{\I}^\top = \PP_{\I}$.
Notably,  for all $i,j\in \I$, we have $\PP_{\I} v_i = \PP_{\I} v_j$.
Since the subdifferential of the atomic norm satisfies $F_{\I}= \mathrm{conv}\{v_i\colon i\in \I\}$, and because $\PP_{\I} v_i =\PP_{\I} v_{i_0}$ for a fixed $i_0\in \I$, it follows that
\[
\PP_{\I} \,\mathrm{conv}\{v_i\colon i\in\I\} = \{ \PP_{\I} v_{i_0}\}.
\]
\begin{definition}\label{def:f}
For $\I\in\mathcal I$, let $f_{\I}:=\PP_{\I}v_{i_0}$, where $i_0\in\I$ is arbitrary. 
Equivalently, $\PP_{\I}F_{\I}=\{f_{\I}\}$. We call $f_{\I}$ the projection of the face $F_{\I}$ onto $\SS_{\I}$.
\end{definition}

Note that in general $f_\I$ may not belong to $F_\I=\mathrm{conv}\{v_i\colon i\in \I\}$.  The next subsection analyses this issue in detail.
Although $f_{\I}\in\SS_{\I}$, it need not belong to $F_{\I}$;
see Example~\ref{ex:simple}.
The condition $f_{\I}\in\mathrm{ri}(F_{\I})$ will be important for
pattern recovery under the irrepresentability condition
\eqref{eq:irrep0}; see Lemma~\ref{lem:tau_positive}.

\subsubsection{Examples of patterns}\label{sec:ex-patt}

In this subsection, we give two examples how patterns arise from given atomic norms.

\begin{example}[$\ell_1$-norm]\label{ex:l1}
Consider the $\ell_1$-norm $\|\cdot\|_1$ on $\R^m$. We note that $\poly{\cdot}=\|\cdot\|_1$ holds for $\{v_1,\ldots,v_K\}=\{-1,1\}^m$; in particular $K=2^m$.
Then, the dual ball is $B^\ast=[-1,1]^m$. 

The active set  $\I_x$ (after the identification of indices with their corresponding vectors, $i\leftrightarrow v_i$), for $x\neq 0$, equals 
\[
\I_x = \{v\in\{-1,1\}^m\colon x_j\neq 0 \implies v_j=\mathrm{sign}(x_j)\mbox{ for }j\in\{1,\ldots,m\}\}.
\]
Each face of $B^\ast$ can be written as 
\[
F_{x} =\left\{ y\in[-1,1]^m\colon x_j\neq0\implies y_j = \mathrm{sign}(x_j)\mbox{ for }j\in\{1,\ldots,m\}\right\}.
\]
Further, we have 
\[
C_x = \{y\in\R^m\colon \mathrm{sign}(x_j)=\mathrm{sign}(y_j)\mbox{ for }j\in\{1,\ldots,m\}\},
\]
where we use the convention that $\mathrm{sign}(0)=0$. 
Then, the corresponding pattern subspace is given by 
\begin{align*}
\mathcal{S}_{x} &=\{y\in\R^m\colon x_j=0 \implies y_j=0\mbox{ for }j\in\{1,\ldots,m\}\}.
\end{align*}
The orthogonal projection $\PP_{\I_x}$ onto $\SS_{x}$ is a diagonal matrix with $(\PP_{\I_x})_{jj} = 1$ if $x_j\neq 0$ and $(\PP_{\I_x})_{jj} = 0$ otherwise. 

The projection of $F_x$ onto $\SS_x$ is
\[
f_{\I_x}=\mathrm{sign}(x)
=\bigl(\mathrm{sign}(x_j)\bigr)_{j=1,\ldots,m}.
\]
This explains why $\ell_1$-norm induces element-wise sparsity. 
\end{example}

\begin{example}[$\ell_\infty$-norm]\label{ex:linfty}
Consider the $\ell_\infty$-norm on $\R^m$, $\|x\|_\infty = \max_{i=1,\ldots,m} |x_i|$. 
This is an atomic norm $\poly{x} = \|x\|_\infty$ with the set of vectors $\{v_i\}_{i=1}^K = \{e_1, -e_1, \ldots, e_m, -e_m\}$, so $K=2m$, where $\{e_1,\dots,e_m\}$ is the standard basis of $\R^m$.
The dual norm is $\|y\|_\infty^* = \|y\|_1$, and the dual unit ball is $B^* = \mathrm{conv}\{ \pm e_1, \ldots, \pm e_m\}$.

For $x \neq 0$, let $J_x = \{ j \in \{1, \ldots, m\} \colon |x_j| = \|x\|_\infty \}$ be the set of indices where the maximum absolute value is attained. 
The active set of vectors is $\I_x$ corresponding to $\{ \mathrm{sign}(x_j) e_j \colon  j \in J_x \}$.

The active set $\I_x$ can be characterized using $J_x$. Let $M=\|x\|_\infty$. An index $k$ is in $\I_x$ if either:
\begin{itemize}
	\item $k = 2j-1$ for some $j \in J_x$, and $v_{2j-1}^\top x = e_j^\top x = x_j = M$. This requires $x_j > 0$.
	\item $k = 2j$ for some $j \in J_x$, and $v_{2j}^\top x = (-e_j)^\top x = -x_j = M$. This requires $x_j < 0$.
\end{itemize}
So, $\I_x = \{ 2j-1 \mid j \in J_x, x_j = M \} \cup \{ 2j \mid j \in J_x, x_j = -M \}$.

The subdifferential $\partial \|x\|_\infty$, which is a face $F_x$ of $B^*$, is:
\[
F_x = \mathrm{conv}\{ \mathrm{sign}(x_j) e_j \colon j \in J_x \}.
\]
Let $s_x$ be the vector with $(s_x)_j = \mathrm{sign}(x_j)$ for $j \in J_x$ and $0$ otherwise. We have $|J_x| = \|s_x\|_1$. Let $S^c = \{1, \ldots, m\} \setminus J_x$.  The pattern subspace is $\SS_{\I_x} = \mathrm{span}(\{s_x\} \cup \{e_j \mid j \in S^c\})$. The pattern equivalence class is
\[
C_x=\left\{a s_x+\sum_{k\in S^c}y_k e_k\colon
a>0,\ |y_k|<a\text{ for every }k\in S^c\right\}.
\]
Thus, the pattern records the indices and signs of the coordinates
attaining the largest absolute value.

The orthogonal projection $\PP_{\I_x}$ onto $\SS_{\I_x}$ is given by:
\[
(\PP_{\I_x} z)_k =
\begin{cases}
	\frac{ s_x^\top z }{\|s_x\|_1} \mathrm{sign}(x_k), & \mbox{if } k \in J_x, \\
	z_k, & \mbox{if } k \in S^c.
\end{cases}
\]

The projection $f_{\I_x} = \PP_{\I_x} v_{i_0}$ is $f_{\I_x} = s_x/\|s_x\|_1$.
 \end{example}

\subsection{Restricted dual norm}\label{sec:restricted}

The restricted dual gauge arises naturally when the optimization
problem is restricted to a pattern subspace.
This restriction is a key step in the primal-dual witness
method, see \eqref{eq:tildeK}.

Fix non-empty $\I\in\mathcal{I}$. For $y\in \SS_{\I}$, define the restricted dual gauge as
\[
\polydr{y}{\I} = \sup_{\substack{x\in \SS_{\I}\\ \poly{x}\leq 1}} x^\top y.
\]
The general properties of dual and restricted dual gauges are presented in the Appendix \ref{app:dg}. 

\begin{lemma}\label{lem:restr_ball}
Let $B_{\I}^\ast$ be the unit ball in $\SS_{\I}$ with respect to
 $\polydr{\cdot}{\I}$. Then,
\[   
B_{\I}^\ast:= \{ y\in\SS_{\I}\colon \polydr{y}{\I}\leq 1\}  = \mathrm{conv}\{\PP_{\I} v_1,\ldots, \PP_{\I} v_K\}. 
\]
\end{lemma}
It is clear that if $\I\neq\emptyset$, then the projection $f_\I$ of the face $F_{\I}$ belongs to the boundary of the unit restricted dual norm ball so that 
$\polydr{f_\I}{\I}=1$. 
For any $y\in\SS_{\I}$, we have $\polyd{y} \geq \polydr{y}{\I}$, which implies that $\polyd{f_\I}\geq 1$ and it is possible that $\polyd{f_\I}>1$ as we show in the example below.

\begin{example}\label{ex:simple}
The following example studies a simple atomic norm for $m=2$ and $K=4$. Suppose that 
\[
(v_1,\ldots,v_K) = \left((\beta,1-\alpha)^\top,
(\beta,-1-\alpha)^\top,
(-\beta,-1+\alpha)^\top,(-\beta,1+\alpha)^\top\right),
\]
if $\alpha\in\R$ and $\beta\neq 0$, then $\poly{x}$ defines a norm on $\R^2$. 
Note that these are the vertices of the dual norm ball $B^\ast$,
while the vertices of the norm ball $B$ are $\{\pm\frac1\beta(1,0),\,\pm\frac1\beta(\alpha,\beta)\}$.

The non-trivial active sets are $\I\in \{\{1,2\}, \{2,3\}, \{3,4\},\{1,4\}\}$. Denote $d_{i,j}=v_i-v_j$. We have $\SS_{\{i,j\}} = \{y\in\R^2\colon d_{i,j}^\top y=0\}$. The orthogonal projector onto $\SS_{\{i,j\}}$ is given by $\PP_{\{i,j\}} = I_2 - d_{i,j}d_{i,j}^\top / \|d_{i,j}\|_2^2$. 
We have 
\begin{align*}
    f_\I = \begin{cases}
        (\beta,0)^\top, & \I=\{1,2\},\\
         -\frac{\beta}{\alpha^2+\beta^2} (\alpha,\beta)^\top, & \I=\{2,3\},\\
            (-\beta,0)^\top, & \I=\{3,4\},\\
         \frac{\beta}{\alpha^2+\beta^2} (\alpha,\beta)^\top, & \I=\{1,4\},
    \end{cases}
\end{align*}
and $\polyd{f_\I} = \max\{|\alpha|,1\}$ for $\I\in\{ \{1,2\},\{3,4\}\}$ and $\polyd{f_\I} = \max\left\{\frac{|\alpha|}{\alpha^2+\beta^2},1\right\}$ otherwise. 
Moreover,
\[
f_\I\in\mathrm{ri}(F_\I) \iff \begin{cases}
    \alpha\in(-1,1), & \I\in\{ \{1,2\},\{3,4\}\},\\
    \alpha^2+\beta^2>|\alpha|, & \I\in\{\{2,3\},\{1,4\}\}.
\end{cases}
\]
The unit restricted dual balls $B_{\I}^\ast$ are the line segments $[f_{\{1,2\}},f_{\{3,4\}}]$ and $[f_{\{2,3\}},f_{\{1,4\}}]$.
The dual balls with the projections $f_{\I}$ are depicted in Figure \ref{fig:ex}.

\newcommand{\drawDualAndShadows}[3]{%
  \begingroup
    \def\a{#2}\def\b{#3}%
    \pgfmathsetmacro{\Rx}{1.7}
    \pgfmathsetmacro{\Ry}{4.1}

    \pgfmathsetmacro{\maxx}{abs(\b)}
    \pgfmathsetmacro{\tA}{abs(1-\a)}
    \pgfmathsetmacro{\tB}{abs(-1-\a)}
    \pgfmathsetmacro{\tC}{abs(-1+\a)}
    \pgfmathsetmacro{\tD}{abs(1+\a)}
    \pgfmathsetmacro{\maxy}{max(\tA,\tB,\tC,\tD)}

    \pgfmathsetmacro{\s}{min(\Rx/\maxx,\Ry/\maxy)}

    \begin{tikzpicture}[scale={#1}]
      \path[use as bounding box]
        (-\Rx,-\Ry) rectangle (\Rx,\Ry);

      \draw[->] (-\Rx-0.3,0) -- (\Rx+0.5,0) node[below right] {};
      \draw[->] (0,-\Ry) -- (0,\Ry) node[left]        {};

      \coordinate (v1) at ({\s*\b},{\s*(1-\a)});
      \coordinate (v2) at ({\s*\b},{\s*(-1-\a)});
      \coordinate (v3) at ({-\s*\b},{\s*(-1+\a)});
      \coordinate (v4) at ({-\s*\b},{\s*(1+\a)});

      \draw[thick,black] (v1) -- (v2) -- (v3) -- (v4) -- cycle;

      \coordinate (c12) at ({\s*\b},0);
      \coordinate (c34) at ({-\s*\b},0);
      \pgfmathsetmacro{\den}{\a*\a + \b*\b}
      \coordinate (c23) at ({-\s*(\b*\a/\den)},{-\s*(\b*\b/\den)});
      \coordinate (c14) at ({ \s*(\b*\a/\den)},{ \s*(\b*\b/\den)});


      \foreach \pt/\lab/\pos in {
        c12/$f_{12}$/below right,%
        c23/$f_{23}$/below left,%
        c34/$f_{34}$/left,%
        c14/$f_{14}$/above right%
      }{
        \fill[red] (\pt) circle(2pt);
        \node[red,\pos] at (\pt) {\lab};
      }

      \draw[dashed, thick] ({1/\b},0) -- ({\a/\b},1)--({-1/\b},0)--({-\a/\b},-1)--cycle;
      \draw[dotted,blue, thick] ({-\Rx-0.3},{-(\Rx+0.3)*\b/\a})--({\Rx+0.3},{(\Rx+0.3)*\b/\a});
    \end{tikzpicture}
  \endgroup
}

\begin{figure}[htbp]
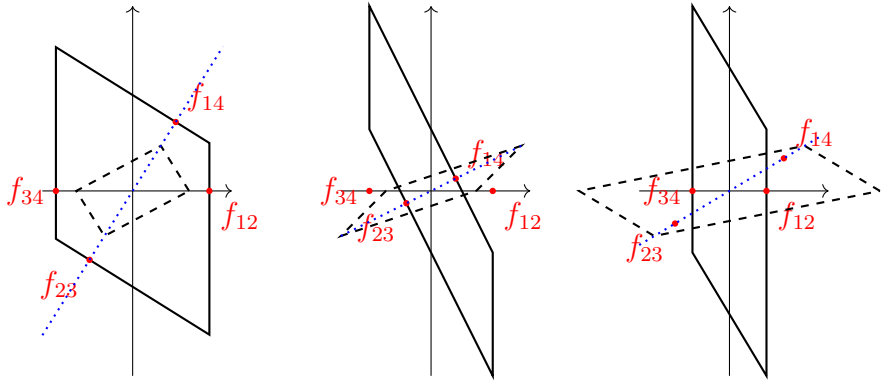

\centering
\begin{tabular}{c@{\qquad}c@{\qquad}c}
  \drawDualAndShadows{0.6}{0.5}{0.8} &
  \drawDualAndShadows{0.6}{2}{1} & 
  \drawDualAndShadows{0.6}{0.5}{0.3} 
\end{tabular}
  \caption{Solid parallelogram: the dual ball $B^\ast$, dashed parallelogram: the ball $B$, blue dotted line: $\SS_{\{1,4\}}=\SS_{\{2,3\}}$, red dots: projections $f_{\I}$. Note that $\SS_{\{1,2\}}=\SS_{\{3,4\}}$ is the $x$-axis.
    Left: $(\alpha,\beta)=(0.5,0.8)$. 
    Middle: $(\alpha,\beta)=(2,1)$. 
    Right: $(\alpha,\beta)=(0.5,0.3)$. 
    The figures are rescaled.
  }\label{fig:ex}
  \end{figure}

\end{example}

\subsection{Vectorization}\label{sec:vectorization}

The penalty \eqref{eq:penK} acts on the distinct off-diagonal entries,
whereas the loss function acts on full matrices.
The following notation connects the vector geometry above
with the precision-matrix estimation problem.

Let $p\in\N$. Let $m=p(p-1)/2$.  Let $\mathrm{Sym}(p)$ denote the space of symmetric real $p\times p$ matrices, and let  $\mathrm{Sym}_+(p)$  be the cone of positive definite matrices within $\mathrm{Sym}(p)$. We equip $\mathrm{Sym}(p)$ with the trace inner product, $\scalar{X}{Y}=\tr{XY}$. For  $X\in\mathrm{Sym}(p)$, define the vectorization operators:
\begin{itemize}
    \item $\v{X}\in\R^{p^2}$ obtained by stacking the columns of $X$ into a single column vector,
    \item $\vp{X}\in \R^{m}=\vp{\mathrm{Sym}(p)}$ obtained by stacking only the strict lower triangular entries  $(X_{ij})_{i> j}$, excluding diagonal elements. 
\end{itemize}
For $a,b,c\in \mathrm{Mat}(p\times p)$, we have 
$\v{a b c}=(c^\top\otimes a)\v{b}$, 
where $\otimes$ denotes the Kronecker product.

Let $\mathrm{Sym}_0(p)$ be the zero-diagonal subspace of $\mathrm{Sym}(p)$. 
Let $D\in\mathrm{Mat}(p^2\times m)$ be the duplication matrix defined by
$D\,\vp{X}=\v{X}$ for every $X\in\mathrm{Sym}_0(p)$.

The auxiliary norm $\cnorm{\cdot}$ on $\R^{p^2}$ measures
full-matrix perturbations through $\cnorm{\v{X}}$ and
off-diagonal perturbations through $\cnorm{D\pi}$,
 $\pi\in\R^m$.
We denote its induced operator norm by $\copnorm{\cdot}$.
For $X\in\mathrm{Sym}(p)$, we write $\patt{X}$ as shorthand
for $\patt{\vp{X}}$.

\subsection{Face-stability threshold}\label{sec:tau}
In proving our main result, we utilize the primal-dual witness method developed in \cite{wain09,ravikumar2011graphical}. We begin by considering a restricted version of the optimization problem in $\v{\mathrm{Sym}(p)}$ and constructing a vector $\pi\in\vp{\mathrm{Sym}(p)}=\R^m$ with $\PP_\I \pi = f_\I$; in particular $\polydr{\PP_{\I}\pi}{\I}=1$.

The key step in the method is then to show that $\polyd{\pi}\leq 1$. To achieve this, we control $\polyd{\pi}$, by bounding the term $\pi^\bot=(I_m-\PP_{\I})\pi$. 
We measure this perturbation by $\cnorm{D\pi^\bot}$,
using the auxiliary norm introduced in the previous subsection.

\begin{definition}\label{def:tau}
Let $\poly{\cdot}$ be an atomic norm and let $\I\neq\emptyset$ be the active set of indices. We define the (face-stability) threshold 
\begin{align*}
\tau_\diamond(\I)& =
\sup\{t\geq0\colon \forall\,\pi^\bot\in\SS_{\I}^\bot\,\,\,\, \cnorm{D\pi^\bot}\leq t\implies \polyd{f_{\I}+\pi^\bot}\leq 1\}.
\end{align*}
\end{definition}

\begin{remark}\label{rem:singleton0}
If $|\I|=1$, then $F_\I=\{f_\I\}$ and
$\SS_\I=\mathbb R^m$, so that $\SS_\I^\perp=\{0\}$.
Consequently, $\tau_\diamond(\I)=+\infty$.
Thus, in this case there are no off-pattern directions to control. See also the paragraph following  \eqref{eq:hatK2} and \eqref{eq:tildeK} in Section \ref{sec:pdwm}.
\end{remark}

When $|\I|\geq2$ and $f_{\I}\in F_{\I}$,
$\tau_\diamond(\I)$ is the largest radius such that every perturbation
$\pi^\perp\in\SS_{\I}^\perp$ with
$\cnorm{D\pi^\perp}\leq\tau_\diamond(\I)$ leaves
$f_{\I}+\pi^\perp$ in $F_{\I}$; see Figure~\ref{fig:tau}.

\begin{figure}[htbp]
    \centering
    \begin{tikzpicture}[decoration={brace, amplitude=5pt}]
        \draw[-latex] (-1.5,0) -- (1.5,0) node[right] {\tiny $\pi_1$};
        \draw[-latex] (0,-1.5) -- (0,1.5) node[above] {\tiny $\pi_2$};
        \draw (1,-0.5) -- (1,0.5) -- (0.5,1) -- (-0.5,1) -- (-1,0.5) -- (-1,-0.5) -- (-0.5,-1) -- (0.5,-1) -- cycle;
        \draw[thick] (-1.25,-1.25) -- (1.25,1.25);
        \draw[decorate, decoration={brace,mirror, amplitude=5pt}]
             (-1,-0.5) -- (-0.75,-0.75)
             node[midway, anchor=north east] {\tiny $\tau_\diamond(\I)$};
        \fill[red] (-0.75,-0.75) circle (2pt) node[right]{$f_I$};
        \draw[dashed] (0,-1.5) -- (-1.5,0) node[midway, xshift=3mm, yshift=-6mm] {\tiny $\SS_{\I}^\perp$};
            \end{tikzpicture}
            \caption{Consider a SLOPE norm with $w=(1,1/2)$ and a pattern $I=(-1,-1)$.
            The thick black line corresponds to $\mathcal{S}_\I = \R(1,1)^\top$.
            The red dot represents the projection $f_\I$.}\label{fig:tau}
\end{figure}
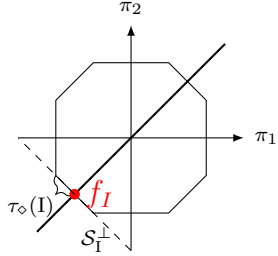

\begin{remark}\label{rem:tau_does_not_exist}
Let $\I=\I_x$ for some $x\neq0$. Recall that in general we have $1=\polydr{f_\I}{\I} \leq \polyd{f_\I}$.
The threshold $\tau_\diamond(\I)$ is defined when the set in
Definition~\ref{def:tau} is nonempty. This is equivalent to the
admissibility of $t=0$, and hence to $\polyd{f_{\I}}=1$.
\end{remark}

\begin{lemma}\label{lem:tau}
If $\polyd{f_\I}=1$, then, with the convention $\inf\emptyset=+\infty$,
\[
\tau_\diamond(\I)=  \inf_{\pi^\bot\in\SS_\I^\bot\colon \polyd{f_\I+\pi^\bot}> 1}\cnorm{D\pi^\bot}.
\]
\end{lemma}

In general, it is possible that $\tau_\diamond(\I)=0$ for some active set $\I$ and the following result characterizes this condition. This happens exactly when $f_\I$ belongs to the boundary of $\mathrm{ri}(F_\I)$. In Example \ref{ex:simple}, such a case occurs exactly if $\alpha\in\{-1,1\}$ or $\alpha^2+\beta^2=|\alpha|$.

\begin{lemma}\label{lem:tau_positive}
$\qquad\qquad\qquad \tau_\diamond(\I)>0 \quad\iff\quad f_\I\in\mathrm{ri}(F_\I)$.
\end{lemma}

For $\I=\I_x$ with $x\neq0$, the condition
$f_\I\in\mathrm{ri}(F_\I)$ implies $\polyd{f_\I}=1$.

\subsection{Pattern-stability threshold}\label{sec:zeta}
Recall the definition of $\SS_x$ from Section \ref{sec:pattern}.
The second threshold specifies a sufficient condition under which an element $x$ in the pattern subspace $\SS_y$ shares the same pattern as $y$, $\patt{x}=\patt{y}$. 
Concretely, it does so by imposing a bound on the distance between $x$ and $y$. 

Let $K^\ast\in\mathrm{Sym}_+(p)$ be the true precision matrix with
$\vp{K^\ast}\neq0$. Let $\Sast=\SS_{\vp{K^\ast}}$ be the pattern
subspace associated with $\vp{K^\ast}$, and let
$\Mast=\{K\in\mathrm{Sym}(p)\colon\vp{K}\in\Sast\}$ denote the
corresponding matrix-pattern subspace.
Set $\Gamma^\ast=(K^\ast)^{-1}\otimes(K^\ast)^{-1}$.

\begin{definition}\label{def:zeta}
Define the pattern-stability threshold
\begin{align*}
\zeta_\diamond(K^\ast)
= \sup\Bigl\{ z>0 \colon &\forall\,K \in \Mast,\\ 
&\cnorm{\Gamma^\ast\v{K-K^\ast}} \le z 
\implies \patt{K}=\patt{K^\ast}\Bigr\}.
\end{align*}
\end{definition}

We now show that this supremum can be characterized as the distance from $x$ to the complement of its pattern-equivalence class in $\SS_x$: 
\begin{lemma}\label{lem:zeta} We have
\begin{align*}
\zeta_\diamond(K^\ast) 
 = 
\inf_{\substack{K\in\Mast\\
      \patt{K}\neq\patt{K^\ast}}} \cnorm{\Gamma^\ast\v{K-K^\ast}}.
\end{align*}
Moreover, if $\cnorm{\cdot}=\|\cdot\|_\infty$,  $x^\ast=\vp{K^\ast}$ and $\SS^\ast=\SS_{x^\ast}$, then 
\begin{align*}
\zeta_\diamond(K^\ast) 
 \geq  \opnorm{K^\ast}_\infty^{-2}\!\!\!\inf_{\substack{x\in\SS^\ast\\
      \patt{x}\neq\patt{x^\ast}}} \!\!\! \|x-x^\ast\|_\infty.
\end{align*}
\end{lemma}
Here, $\opnorm{\cdot}_{\infty}$ denotes the operator norm induced by the vector $\ell_\infty$-norm.

\section{Main results}\label{sec:main}

\subsection{Setting and notation}\label{sec:notation}

We return to the estimator introduced in Section~\ref{sec:setup}:
\begin{align}\label{eq:hatK}
\hat K
=
\argmin_{K\in\mathrm{Sym}_+(p)}
\left\{
\tr{\hat\Sigma K}-\log\det(K)+\lambda\norm(K)
\right\}.
\end{align}
Our aim is to give sufficient conditions under which
$\hat K$ has the same pattern as $K^\ast$.
We first introduce the quantities entering these conditions.

Recall that the penalty is given by $\norm$ on $\mathrm{Sym}(p)$ by
$\norm(X) = 2 \poly{ \vp{X}}$,
where $\poly{\cdot}$ is an atomic norm on $\R^m=\vp{\mathrm{Sym}(p)}$, $m=p(p-1)/2$.  
The dual gauge associated with $\norm$ on $\mathrm{Sym}_0(p)$ is given by 
\begin{align*}
	\normd(Y) &= \sup_{\substack{X\in \mathrm{Sym}_0(p) \\ \norm(X)\leq 1}}  \tr{Y X}   = \polyd{\vp{Y}},\qquad Y\in\mathrm{Sym}_0(p). 
\end{align*}

Let $\Sigma^\ast\in\mathrm{Sym}_+(p)$ and $K^\ast=(\Sigma^\ast)^{-1}$ be the true covariance and precision matrices, respectively.  We assume throughout that $\vp{K^\ast}\neq0$. 
Let $\Sast = \SS_{\vp{K^\ast}}$ be the pattern subspace of the true precision matrix and let $\I^\ast$ denote the corresponding active set so that $\Sast=\SS_{\I^\ast}$.  We also define the matrix pattern subspace of  $\mathrm{Sym}(p)$  by
$\Mast = \{ K\in\mathrm{Sym}(p)\colon \vp{K} \in \Sast\}$.

Define the $p^2\times p^2$ matrix $\Gamma^\ast=\Sigma^\ast\otimes\Sigma^\ast$ and its block components 
\[
\Gamma^\ast_{1,1} =  \PM \Gamma^\ast \PM\quad\mbox{and}\quad\Gamma^\ast_{2,1} =  (I_{p^2}-\PM) \Gamma^\ast \PM,
\]
where $\PM\in\mathrm{Sym}(p^2)$ is the orthogonal projection matrix from $\v{\mathrm{Sym}(p)}$ onto $\v{\Mast}$. 

We denote by $(\Gamma^\ast_{1,1})^+$ the Moore-Penrose inverse of $\Gamma^\ast_{1,1}$ and define the matrix
\[
\Q=\Gamma^\ast(\Gamma_{1,1}^\ast)^{+}.
\]
Applying properties of the Moore-Penrose inverse yields
$\Q - \PM = \Gamma_{2,1}^\ast (\Gamma_{1,1}^\ast)^{+}$.

We use the auxiliary norm $\cnorm{\cdot}$ and its induced
operator norm $\copnorm{\cdot}$ introduced in
Section~\ref{sec:vectorization}.

\begin{definition}[Irrepresentability condition]\label{def:irrep}
There exists $\alpha\in(0,1)$ such that 
\begin{align}\label{eq:irrep0}
	\cnorm{ \Gamma_{2,1}^\ast (\Gamma_{1,1}^\ast)^{+}D f_{\I^\ast}}\leq (1-\alpha)\tau_\diamond(\I^\ast), 
\end{align}
where $f_{\I^\ast}$ is the projection of $F_{\I^\ast}$ and $D$ is the duplication matrix.
\end{definition}
\begin{remark}
If $|\I^\ast|=1$, then $\Sast=\mathbb R^m$ and
$\Mast=\mathrm{Sym}(p)$. Thus, $\Gamma_{2,1}^\ast
=
0$.
Hence, the left-hand side of \eqref{eq:irrep0} is zero, whereas
its right-hand side is \(+\infty\), because
\(\tau_\diamond(\I^\ast)=+\infty\). Consequently, the
irrepresentability condition is satisfied automatically when
$|\I^\ast|=1$ and imposes a genuine restriction only when
$|\I^\ast|\geq2$.
\end{remark}

For a general polyhedral atomic norm, condition~\eqref{eq:irrep0} is a penalty-dependent geometric analogue of the classical irrepresentability condition. 

\begin{definition}\label{def:c}
	Let $c_\diamond$ be a constant for which
$\cnorm{D \pi} \leq c_\diamond \polydr{\pi}{\I^\ast}$ holds for all $\pi\in\Sast$.

	Additionally, let $\cH$ be a constant for which 
$\copnorm{\Sigma^\ast\Delta\otimes I_p} \leq \cH \cnorm{\Gamma^\ast\v{\Delta}}$ holds for all $\Delta\in \Mast$.
\end{definition}
If $\cnorm{\cdot}$ is the $\ell_\infty$-norm, then one can take $\cH = \|\v{K^\ast}\|_1=\sum_{i,j}|K^\ast_{ij}|$ (see Lemma \ref{lem:kappainfty}). 

\subsubsection{Summary of notation}
For convenience, Table~\ref{tab:notation} summarizes the principal notation used in the pattern-recovery results and the primal-dual witness analysis.
    \begin{table}[h]
\centering
\footnotesize
\setlength{\tabcolsep}{4pt}
\renewcommand{\arraystretch}{0.99}
\begin{tabular}{@{}p{0.21\linewidth}p{0.75\linewidth}@{}}
\toprule
Symbol & Meaning and role \\
\midrule
$\poly{\cdot},\polyd{\cdot}$
& Atomic norm (gauge) and its dual gauge. \\

$\I_x,F_\I,C_\I$
& Active set, corresponding face, and pattern equivalence class;
equivalent geometric descriptions of a pattern. \\

$\SS_\I,\PP_\I,f_\I$
& Pattern subspace, its orthogonal projector, and the projection of
$F_\I$ onto $\SS_\I$; used to separate within- and off-pattern
directions. \\

$\v{X},\vp{X},D$
& Full vectorization, strict-lower-triangular vectorization, and
duplication matrix. \\

$\Sast,\Mast, \PM$
& True vector and matrix pattern subspaces and the corresponding
projector.  \\

$\Gamma^\ast,\Q$
&  Hessian $\Gamma^\ast=\Sigma^\ast\otimes\Sigma^\ast$
and the associated projection operator. \\

$\cnorm{\cdot},\copnorm{\cdot},c_\diamond,\cH$
& Auxiliary norm, induced operator norm, and comparison constants
used to control deviations and remainder terms. \\

$\tau_\diamond,\zeta_\diamond$
& Face- and pattern-stability thresholds; control off-pattern dual
feasibility and preservation of the pattern within the pattern
subspace, respectively. \\
$\alpha$ & parameter in the irrepresentability condition;
quantifies the margin available for controlling off-pattern
perturbations.\\
\bottomrule
\end{tabular}
\caption{Principal notation used in the pattern-recovery results and the primal-dual witness analysis.}
\label{tab:notation}
\end{table}

\subsection{Main Theorem}\label{subsec:main}

Although we do not explicitly assume $f_{\I^\ast}\in\mathrm{ri}(F_{\I^\ast})$, the following result requires that $\tau_\diamond(\I^\ast)>0$, which through Lemma \ref{lem:tau_positive} implies that $f_{\I^\ast}\in\mathrm{ri}(F_{\I^\ast})$.  This result is supplemented with the study of the case when $\tau_\diamond$ does not exist, see Section \ref{app:pr}.

The thresholds $\tau_\diamond$ and $\zeta_\diamond$ are defined in Definitions \ref{def:tau} and \ref{def:zeta}, respectively. The following theorem concerns the nontrivial case $|\I^\ast|\geq2$. The singleton case is discussed separately in Remark~\ref{rem:singleton}.

\begin{theorem}\label{cor:main} 
Assume that $|\I^\ast|\geq2$, that $\tau_{\diamond}>0$ and that there exists $\alpha>0$ such that \eqref{eq:irrep0} holds. 
Define  
\[
M = \frac{ \alpha\tau_\diamond}{\copnorm{\Q}( \copnorm{I_{p^2}-\Q} c_\diamond+\alpha\tau_\diamond)}\quad\mbox{and}\quad r = 1- \frac{1}{\sqrt{1+M}}
\]
and
\[
     \delta = 
       \left(\sqrt{1+M}-1\right)^2/\cH\quad \mbox{if }r \leq \cH \zeta_\diamond\quad\mbox{and}\quad \delta=M \zeta_\diamond-\frac{\cH\zeta_\diamond^2}{1-\cH\,\zeta_\diamond}\quad\mbox{otherwise}.
\]
Let $\hat{K}$ be the unique solution to \eqref{eq:hatK} with
\[
\lambda=\frac{\min\{r, \eta \zeta_{\diamond}\} \copnorm{I_{p^2}-\Q}}{\cH\copnorm{\Q}(  c_\diamond \copnorm{I_{p^2}-\Q}+\alpha\tau_\diamond)}
\]
If 
$\cnorm{\v{\hat{\Sigma}-\Sigma^\ast} }< \delta$,
then $\patt{\hat{K}}=\patt{K^\ast}$ and
\[
\cnorm{\Gamma^\ast\v{\hat{K}-K^\ast}}\leq \frac{1}{\cH}\left(1-\frac{1}{\sqrt{1+M}}\right). 
\]
\end{theorem}

In practice, both the tuning parameter and the deviation bounds are defined in terms of population quantities, which are typically unknown. For some specific models, however, these quantities can be computed explicitly or approximated with high accuracy. A prime example is the use of the $\ell_1$ penalty in GLASSO, where the underlying geometry is simple enough to render the key quantities explicit; see Theorem~\ref{thm:glasso}.

The constant $\delta$ defined in Theorem \ref{cor:main} is always positive. Moreover, one has the following asymptotic expansions as $M\to 0$:
    \begin{align*}
\left(\sqrt{1+M}-1\right)^2  =\frac{M^2}{4} + O(M^3)\quad\mbox{and}\quad
1-\frac{1}{\sqrt{1+M}} = \frac{M}{2} + O(M^2). 
\end{align*}

\begin{remark}\label{rem:singleton}
If $|\I^\ast|=1$, then $\Sast=\R^m$ and $\tau_\diamond(\I^\ast)=+\infty$.
Consequently, there are no off-pattern directions to control,
and the irrepresentability condition is automatic.
In this case, set
$M=1/\copnorm{\Q}$ and define $r$ and $\delta$ by the formulas in Theorem~\ref{cor:main}, using this value of $M$.

Assume $\cnorm{\v{\hat\Sigma-\Sigma^\ast}}<\delta$ and take any $\lambda>0$ satisfying
\[
\cnorm{\v{\hat\Sigma-\Sigma^\ast}}
+\lambda\, c_\diamond<\delta.
\]
Then the unique solution $\hat K$ to \eqref{eq:hatK} satisfies $\patt{\hat K}=\patt{K^\ast}$ and $\cnorm{\Gamma^\ast\v{\hat{K}-K^\ast}}\leq r/\eta$.

The proof follows from the same argument as that of
Theorem~\ref{cor:main}, with the dual-feasibility step omitted, see Theorem~\ref{thm:patt_recov}.
\end{remark}

An appealing direction for applications of this result is the construction of penalties tailored to maximize the probability of recovering a given pattern. This amounts to designing a norm $\poly{\cdot}$ for which the corresponding value of $\delta$ is as large as possible. In the case of the SLOPE penalty $\|\cdot\|_w$, we pursue this idea in Section \ref{sec:SLOPE} to construct weights $w_i$ that are optimal for recovering patterns with a single non-null cluster; see also numerical experiments in Section~\ref{sec:num}. We also show that the OSCAR penalty \cite{OSCAR08} with suitably chosen parameters is the optimal choice (among all SLOPE penalties) when there is no information about the true pattern.

We consider three natural choices for the norm $\cnorm{\cdot}$ on $\R^{p^2}$. 
\begin{itemize}
	\item  Norm based on a dual norm: for $X\in\mathrm{Mat}(p\times p)$, define
	\[
	\cnorm{\v{X}} =\max\left\{\polyd{\vp{X}}, \polyd{\vp{X^\top}}, \|(X_{i,i})_{i=1}^p\|_\infty\right\}.
	\]
	This choice aligns the measure $\tau_\diamond$ of the perturbation $\pi^\bot$	using the same norm that defines the feasible region. 
Consequently, $\tau_\diamond$ becomes intrinsically linked to the geometry of the dual ball $B^\ast$.
	\item $\ell_\infty$-norm: for the specific case of GLASSO, the operator norm defined above coincides with the $\ell_\infty$-norm. This makes the $\ell_\infty$-norm a particularly natural choice, as adopted by \cite{ravikumar2011graphical}. 
\end{itemize}

A third choice is suggested by the decomposition underlying the
primal--dual witness argument. The projection $\Q$ appears in the
fixed-point argument controlling the restricted estimator, whereas
$I_{p^2}-\Q$ controls the off-pattern perturbations; see
Section~\ref{sec:ProofSkeleton}. This leads
to the projection-adapted maximum norm
\begin{equation}\label{eq:projection-adapted-norm}
    \|v\|_{\Q}
    :=
    \max\left\{
        \|\Q v\|_\infty,\,
        \|(I_{p^2}-\Q)v\|_\infty
    \right\},
    \qquad v\in\R^{p^2}.
\end{equation}

Since $\Q^2=\Q$, we have 
\[
\|\Q v\|_{\Q}\leq \| v\|_{\Q},
\qquad
\|(I_{p^2}-\Q)v\|_{\Q}
\leq 
\|v\|_{\Q}.
\]
Both projections are therefore contractions in this norm.
Since every nonzero projection has induced operator norm at least one, their contractivity 
 gives
\[
\copnorm{\Q}
=
\copnorm{I_{p^2}-\Q}
=
1.
\]
Thus, 
Theorem~\ref{cor:main} holds with
\[
M=\frac{\alpha\tau_\diamond}
        {c_\diamond+\alpha\tau_\diamond},
\qquad
r=1-\frac{1}{\sqrt{1+M}},
\qquad
\lambda
=
\frac{\min\{r,\cH\zeta_\diamond\}}
     {\cH(c_\diamond+\alpha\tau_\diamond)},
\]
with $\delta$ given by the same formula as in the theorem.

Moreover, the choices $\cnorm{\cdot}=\|\cdot\|_{\Q}$ and $\cnorm{\cdot}=\|\cdot\|_\infty$ yield the same values of
$\tau_\diamond(\I^\ast)$ and $\zeta_\diamond(K^\ast)$,
and the same irrepresentability condition. Indeed, $\|\cdot\|_{\Q}$ coincides with
$\|\cdot\|_\infty$ on both $\ker(\Q)$ and
$\operatorname{ran}(\Q)$.
Since $\Q\PM=\Q$, we have $
 D\pi^\bot$ and $(\Q-\PM)D f_{\I^\ast}\in\ker(\Q)$.
Their norms are therefore the same under the two choices,
so $\tau_\diamond(\I^\ast)$ and the irrepresentability
condition, including its admissible margin $\alpha$,
coincide.

Similarly, for every $\Delta\in\Mast$,
we have $\Q\Gamma^\ast\v{\Delta}=\Gamma^\ast\v{\Delta}$, and hence
\begin{align}\label{eq:GDQ}
\|\Gamma^\ast\v{\Delta}\|_{\Q}
=
\|\Gamma^\ast\v{\Delta}\|_\infty.
\end{align}
It follows that $\zeta_\diamond(K^\ast)$ also coincides
for the two choices of auxiliary norm.
The constants $c_\diamond$ and $\cH$, however, need not
coincide.

\subsection{Finite sample theory and concentration results}\label{sec:FST}

Theorem \ref{cor:main} reduces pattern recovery to controlling $    \cnorm{\v{\hat{\Sigma}_n-\Sigma^\ast}}$, where $    \hat{\Sigma}_n
    =
    \frac{1}{n}\sum_{i=1}^n
    X^{(i)}(X^{(i)})^\top$.
The statistical consequences depend both on the tails of \(X\) and on the norm
\(\cnorm{\cdot}\). To make the latter dependence explicit, define
\[
    L_{\cnorm{\cdot},p}
    :=
    \sup\left\{
        \cnorm{\v{A}}\colon 
        A\in\mathrm{Sym}(p),\
        \|\v{A}\|_\infty\leq 1
    \right\}
\]
so that $\cnorm{\v{A}}
    \leq
    L_{\cnorm{\cdot},p}\|\v{A}\|_\infty$ for all $A\in \mathrm{Sym}(p)$.

Suppose that \(X\) is coordinatewise sub-Gaussian and let
\[
    \nu
    :=
    \max_{1\leq j\leq p}\|X_j\|_{\psi_2}
    <\infty,
    \qquad
    \|Y\|_{\psi_2}
    :=
    \inf\left\{
        s>0:\E[\exp(Y^2/s^2)]\leq 2
    \right\}.
\]
Since \(X_jX_k-\E[X_jX_k]\) is sub-exponential, applying Bernstein's
inequality for centered sub-exponential random variables
\cite[Theorem~2.8.1]{Vershynin}, followed by a union bound over
the \(p^2\) covariance entries, gives
\[
    \P\left(
        \cnorm{\v{\hat{\Sigma}_n-\Sigma^\ast}}\geq t
    \right)
    \leq
    2p^2
    \exp\left[
        -c\,n\min\left\{
            \frac{t^2}
                 {\nu^4L_{\cnorm{\cdot},p}^2},
            \frac{t}
                 {\nu^2L_{\cnorm{\cdot},p}}
        \right\}
    \right],
\]
where \(c>0\) is a universal numerical constant.
Consequently, if \(t=\delta>0\) is the deterministic threshold from
Theorem \ref{cor:main}, we obtain the following result. 
\begin{corollary}[Sub-Gaussian finite-sample recovery] Assume that $X$ is coordinatewise sub-Gaussian with $\max_{1\leq j\leq p}\|X_j\|_{\psi_2}\leq \nu$. 
    Under the assumptions of Theorem \ref{cor:main}, for every $\varepsilon\in(0,1)$, pattern recovery holds with probability at
least $1-\varepsilon$ whenever 
\begin{align} \label{eq:sample-size-general}
    n
    \geq
    c^{-1}\log\left(\frac{2p^2}{\varepsilon}\right)\max
    \left\{
        \frac{\nu^4L_{\cnorm{\cdot},p}^2}{\delta^2}
        ,        \frac{\nu^2L_{\cnorm{\cdot},p}}{\delta}
    \right\}.
\end{align}
\end{corollary}

This formulation applies to every norm $\cnorm{\cdot}$ on \(\R^{p^2}\). For the
entrywise \(\ell_\infty\)-norm,
\(L_{\cnorm{\cdot},p}=1\), and \eqref{eq:sample-size-general} gives the
usual logarithmic dependence on \(p\). For the spectral norm, the generic comparison $L_{\cnorm{\cdot},p}=p$ is generally conservative. Sharper bounds can be obtained directly from matrix-concentration inequalities under suitable vector-level sub-Gaussian or related structural assumptions; see, for example, \cite{CON1,CON4}.

A corresponding explicit statement can also be given under polynomial
moment assumptions. Let \(s>2\) and suppose that
\begin{align}\label{eq:moment2s}
    \nu_{2s}
    :=
    \max_{1\leq j\leq p}
    \E[|X_j|^{2s}]^{1/(2s)}
    <\infty.
\end{align}
By a union bound over the matrix entries and Markov's inequality,
\[
\P( \|\v{\hat{\Sigma}_n-\Sigma^\ast}\|_\infty \geq t) \leq p^2 \frac{\max_{1\leq j,k\leq p}  \|\hat{\Sigma}_{n,jk} - \Sigma^\ast_{jk}\|_{L^s}^s  }{t^s}. 
\]
By the Marcinkiewicz-Zygmund inequality and Minkowski’s inequality, 
\[
 \|\hat{\Sigma}_{n,jk} - \Sigma^\ast_{jk}\|_{L^s} \leq C_s \frac{\| X_{j}X_k - \E[X_j X_k]\|_{L^s}}{\sqrt{n}}, 
\]
where $C_s$ depends only on $s$. Since $\left\|
        X_jX_k-\E[X_jX_k]
    \right\|_{L^s}
    \leq
    2\nu_{2s}^2$,
    we obtain 
 \[
\P( \cnorm{\v{\hat{\Sigma}_n-\Sigma^\ast}} \geq t) \leq p^2 \left( \frac{2 C_s L_{\cnorm{\cdot},p} \nu_{2s}^2}{t \sqrt{n}} \right)^s.
\]
Consequently, taking \(t=\delta\), where \(\delta>0\) is the
deterministic threshold from Theorem~\ref{cor:main}, gives the following
finite-sample guarantee.

\begin{corollary}[Finite-moment finite-sample recovery]
\label{cor:finite-moment-finite-sample}
Let \(s>2\) and assume \eqref{eq:moment2s}.
Under the assumptions of Theorem~\ref{cor:main}, for every
\(\varepsilon\in(0,1)\), pattern recovery holds with probability at
least \(1-\varepsilon\) whenever
\[
    n
    \geq
    \frac{4C_s^2L_{\cnorm{\cdot},p}^2\nu_{2s}^4}{\delta^2}
    \left(\frac{p^2}{\varepsilon}\right)^{2/s}.
\]
\end{corollary}

Thus, sub-Gaussian tails lead to logarithmic dependence on \(p\),
whereas a fixed number of finite moments generally produces a
polynomial dependence for the ordinary sample covariance matrix.
This distinction is consistent with the exponential- and
polynomial-tail regimes considered in
\cite[Section~2.3]{ravikumar2011graphical}.

\subsection{GLASSO specialization and comparison with previous results}\label{sec:glasso}

We now specialize our general theory to the $\ell_1$ penalty and compare the resulting guarantee with the classical graphical lasso analysis of \citet{ravikumar2011graphical}. Our general results extend their framework, but the comparison is most transparent under the stronger irrepresentability condition considered in \citet{ravikumar2011graphical}, see \eqref{eq:irrepglasso} below.

Recall that the matrix operator norm $\opnorm{\cdot}_\infty$ is the maximum absolute row sum.
Let $S=\{(i,j)\colon K^\ast_{ij}\neq 0\}$ 
denote the support of the true precision matrix $K^\ast$, and write $S^c$ for its complement. Let $d$ be the maximum number of nonzero entries in any row of $K^\ast$, and define
$\Theta_{\min}=\min_{(i,j)\in S}|K^\ast_{ij}|$.
Finally, let $\Gamma_{S,S}^\ast$ and $\Gamma_{S^c,S}^\ast$ denote the submatrices of $\Gamma^\ast$ obtained by restricting rows and columns to $S$ and $S^c$ in the obvious way.

\begin{theorem}[$\ell_1$ penalty]\label{thm:glasso}
Assume that $|\I^\ast|\geq 2$ and that, for some $\alpha\in(0,1)$,
\begin{align}\label{eq:irrepglasso}
\opnorm{\Gamma_{S^c,S}^\ast(\Gamma_{S,S}^\ast)^{-1}}_\infty\leq 1-\alpha.
\end{align}
Let $\hat K$ be the unique solution to \eqref{eq:hatK} with
\[
\lambda
=
\frac{\min\left\{\tfrac{\alpha}{4},\,\cH \opnorm{K^\ast}_\infty^{-2}\Theta_{\min}\right\}(1-\alpha/2)}{\cH},
\]
and define
\[
\delta=
\begin{cases}
\dfrac{\alpha^2}{16\,\cH}\left(1-\dfrac{\alpha}{3}\right),
& \mbox{if }\alpha \leq 4\cH \opnorm{K^\ast}_\infty^{-2}\Theta_{\min},\\[1.2ex]
\dfrac{\alpha}{6}\opnorm{K^\ast}_\infty^{-2}\Theta_{\min},
& \mbox{otherwise},
\end{cases}
\]
where $\cH=\min\left\{d\,\opnorm{\Sigma^\ast}_\infty \opnorm{K^\ast}_\infty^2,\ \|\v{K^\ast}\|_1\right\}$ .
If
\[
\|\v{\hat\Sigma-\Sigma^\ast}\|_\infty\leq \delta,
\]
then $\mathrm{sign}(\hat K)=\mathrm{sign}(K^\ast)$
and
\[
\|\v{\hat K-K^\ast}\|_\infty
\leq
\frac{\alpha \opnorm{K^\ast}_\infty^2}{4\,\cH}
=
\frac{\alpha}{4}
\max\left\{
\frac{1}{d\,\opnorm{\Sigma^\ast}_\infty},
\frac{\opnorm{K^\ast}_\infty^2}{\|\v{K^\ast}\|_1}
\right\}.
\]
\end{theorem}

To compare Theorem \ref{thm:glasso} with \citet{ravikumar2011graphical}, we first restate their main deterministic guarantee in an equivalent form. In doing so, we also point out that one step in their proof requires a slightly stronger deviation bound than the one stated in the paper. The corrected form below will serve as the basis for our comparison.

Recall that $\Gamma^\ast=\Sigma^\ast\otimes \Sigma^\ast$, and define
\[
\kappa_{\Sigma^\ast}=\opnorm{\Sigma^\ast}_\infty,
\qquad
\kappa_{\Gamma^\ast}=\opnorm{(\Gamma^\ast_{S,S})^{-1}}_\infty.
\]

\textbf{Theorems 1 and 2 of \citet{ravikumar2011graphical}, restated.}
Assume \eqref{eq:irrepglasso}, and define
\begin{align}\label{eq:29}
\delta_R
=
\left(
6\left(1+\frac{8}{\alpha}\right)
d\,\kappa_{\Sigma^\ast}\kappa_{\Gamma^\ast}
\max\{1,\kappa_{\Sigma^\ast}^2\kappa_{\Gamma^\ast}\}
\right)^{-1}.
\end{align}
Let $\hat K$ be the unique graphical lasso estimator with tuning parameter
$\lambda=\frac{8}{\alpha}\delta_R$.
If
\[
\|\v{\hat\Sigma_n-\Sigma^\ast}\|_\infty\leq \delta_R,
\]
then
\[
\|\v{\hat K-K^\ast}\|_\infty
\leq
2\left(1+\frac{8}{\alpha}\right)\kappa_{\Gamma^\ast}\delta_R.
\]
Moreover, if $(i,j)\notin S$, then $\hat K_{ij}=0$, and if in addition
\begin{align}\label{eq:trueedges}
\Theta_{\min}>
2\left(1+\frac{8}{\alpha}\right)\kappa_{\Gamma^\ast}\delta_R,
\end{align}
then $\mathrm{sign}(\hat K)=\mathrm{sign}(K^\ast)$.

This deterministic formulation is equivalent to the original probabilistic statement in \citet{ravikumar2011graphical}: for every $\tau>2$, there exists $n_0$ such that for all $n\ge n_0$,
\[
\|\v{\hat\Sigma_n-\Sigma^\ast}\|_\infty\leq \delta_R
\]
holds with probability at least $1-p^{-(\tau-2)}$. Under the additional condition \eqref{eq:trueedges}, this yields their full pattern-recovery.

We now discuss the point in the proof that motivates a corrected version of \eqref{eq:29}. First note that $\max\{1,\kappa_{\Sigma^\ast}^2\kappa_{\Gamma^\ast}\}
=
\kappa_{\Sigma^\ast}^2\kappa_{\Gamma^\ast}$.
Indeed,
\[
\kappa_{\Sigma^\ast}^2\kappa_{\Gamma^\ast}
=
\opnorm{\Sigma^\ast}_\infty^2
\opnorm{(\Gamma_{S,S}^\ast)^{-1}}_\infty
\ge
\opnorm{\Gamma^\ast(\PP_S\Gamma^\ast \PP_S)^+}_\infty,
\]
where $\PP_S$ denotes the orthogonal projection onto the vectorized model subspace; cf.\ the definition of $\Q$ in Section \ref{sec:notation}. Since $\Q=\Gamma^\ast(\PP_S\Gamma^\ast \PP_S)^+$ 
is a projection matrix, its operator norm is at least $1$. Hence \eqref{eq:29} reduces to
\begin{align}\label{eq:ravierror}
\delta_R
=
\left(
6\kappa_{\Sigma^\ast}^3\kappa_{\Gamma^\ast}^2
d\left(1+\frac{8}{\alpha}\right)
\right)^{-1}.
\end{align}

Now set $\lambda=8\delta_R/\alpha$. In the proof of \citet[page 963, line 11]{ravikumar2011graphical}, it is stated that
\[
\|\v{R(\Delta)}\|_\infty
\le
\left\{
6\kappa_{\Sigma^\ast}^3\kappa_{\Gamma^\ast}^2
d\left(1+\frac{8}{\alpha}\right)^2\delta_R
\right\}
\frac{\alpha\lambda}{8}.
\]
The goal at that stage is to show $\|\v{R(\Delta)}\|_\infty\le \frac{\alpha\lambda}{8}$,
so that their Lemma 5 may be applied. However, substituting \eqref{eq:ravierror} into the displayed bound yields only
\[
\|\v{R(\Delta)}\|_\infty
\le
\left(1+\frac{8}{\alpha}\right)\frac{\alpha\lambda}{8},
\]
which is insufficient for the desired conclusion. Since $\alpha$ is typically small, the term $(1+8/\alpha)$ can be significant. Thus, that step of the argument does not justify the stated bound \eqref{eq:29}.

A corrected version of Theorems 1 and 2 of \citet{ravikumar2011graphical}
is obtained by strengthening their 
definition of $\delta_R$ to
\begin{align}\label{eq:goodDelta}
\delta_R
=
\left(
6\kappa_{\Sigma^\ast}^3\kappa_{\Gamma^\ast}^2
d\left(1+\frac{8}{\alpha}\right)^2
\right)^{-1}.
\end{align}
With this choice, the displayed bound does imply $
\|\v{R(\Delta)}\|_\infty\le \frac{\alpha\lambda}{8}$, 
and the remainder of the proof goes through unchanged. In other words, \eqref{eq:goodDelta} is the corrected deviation threshold 
that allows to execute
 the proof strategy of \citet{ravikumar2011graphical}.

This discrepancy may propagate to later results that rely directly on the stated version of \citet[Theorem 1]{ravikumar2011graphical}. For example, in \citet[Proposition 11.10]{Wainwright_2019}, the sample-size requirement should be
\[
n>c_0(1+8\alpha^{-1})^4 m^2\log d
\]
rather than
\[
n>c_0(1+8\alpha^{-1})^2 m^2\log d.
\]
Interestingly, in extending this approach to the functional graphical lasso,
\citet{FUN_GLASSO} appear to impose a condition consistent with
\eqref{eq:goodDelta}; see equation~(18) in their Theorem~6.4. They do not,
however, explicitly identify the corresponding issue in the original argument
of \citet{ravikumar2011graphical}.

We are now in a position to compare the two guarantees; see also Section \ref{sec:num} for numerical illustrations. For simplicity, we focus on the part of the bound unrelated to $\Theta_{\min}$. By Theorem \ref{thm:glasso},
\[
\delta
\approx
\frac{\alpha^2}
{16\min\{d\,\kappa_{\Sigma^\ast}\opnorm{K^\ast}_\infty^2,\ \|\v{K^\ast}\|_1\}},
\]
whereas the leading term of the corrected bound \eqref{eq:goodDelta} is
\[
\delta_R
\approx
\frac{1}{384}\,
\frac{\alpha^2}{d\,\kappa_{\Sigma^\ast}^3\kappa_{\Gamma^\ast}^2}
\approx
\delta\cdot \frac{1}{24}
\min\left\{
\left(
\frac{\opnorm{K^\ast}_\infty}{\kappa_{\Sigma^\ast}\kappa_{\Gamma^\ast}}
\right)^2,
\frac{\|\v{K^\ast}\|_1}{d\,\kappa_{\Sigma^\ast}^3\kappa_{\Gamma^\ast}^2}
\right\}.
\]
Since larger values of $\delta$ correspond to stronger pattern-recovery guarantees, Theorem \ref{thm:glasso} improves not only the numerical constant but also, potentially, the asymptotic scaling. The key quantity is
\[
\min\left\{
\left(
\frac{\opnorm{K^\ast}_\infty}{\kappa_{\Sigma^\ast}\kappa_{\Gamma^\ast}}
\right)^2,
\frac{\|\v{K^\ast}\|_1}{d\,\kappa_{\Sigma^\ast}^3\kappa_{\Gamma^\ast}^2}
\right\}.
\]
This term is always bounded above by $\opnorm{\Sigma^\ast}_\infty \opnorm{K^\ast}_\infty$,
and in all numerical examples of Section \ref{sec:num} it is strictly smaller than $1$. The next example shows that it can even be of order $p^{-3}$, which implies a polynomial improvement over \eqref{eq:goodDelta}.

\begin{example}\label{ex:comp}
Let $\rho\in(-1/(p-2),1)\setminus\{0\}$ and define
\[
K^\ast_{ij}=
\begin{cases}
1, & \mbox{if } i=j,\\
\rho, & \mbox{if } i\neq j \mbox{ and } \max\{i,j\}<p,\\
0, & \mbox{if } i\neq j \mbox{ and } \max\{i,j\}=p.
\end{cases}
\]
Then $d=p-1$, and a direct calculation gives
\begin{gather*}
\kappa_{\Gamma^\ast}=(1+(p-2)|\rho|)^2,
\qquad
\opnorm{K^\ast}_\infty=1+(p-2)|\rho|,\\
\kappa_{\Sigma^\ast}
=
\frac{1-\rho+(p-2)(\rho+|\rho|)}
{(1-\rho)(1+(p-2)\rho)},
\qquad
\|\v{K^\ast}\|_1
=
p+|\rho|(p-1)(p-2).
\end{gather*}
Hence, for every fixed $\rho\in(0,1)$, as $p\to\infty$,
\[
\min\left\{
\left(
\frac{\opnorm{K^\ast}_\infty}{\kappa_{\Sigma^\ast}\kappa_{\Gamma^\ast}}
\right)^2,
\frac{\|\v{K^\ast}\|_1}{d\,\kappa_{\Sigma^\ast}^3\kappa_{\Gamma^\ast}^2}
\right\}
\sim
\left(\frac{1-\rho}{2\rho}\right)^3 \frac{1}{p^3},
\]
while the first term inside the minimum is of order $p^{-2}$. Therefore, for positive $\rho$ and sufficiently large $p$,
\[
\delta_R
\approx
\left(\frac{1-\rho}{\rho}\right)^3
\frac{\delta}{192\,p^3}.
\]
\end{example}

\section{Numerical experiments}\label{sec:num}

Consider a simple undirected graph $G=(V,E)$ with $V=\{1,\dots,p\}$. Denote by $A_G$ its adjacency matrix. We represent the edge set $E$ as a subset of  ordered pairs $(i,j)\in V\times V$ with $i>j$.  Throughout this section we set, for $i> j$, 
$K^\ast_{ij} =K^\ast_{ji}= \rho\neq0$ if $(i,j)\in E$ and $K^\ast_{ij}=0$ otherwise.  Thus, we have 
$x^\ast: =\vp{K^\ast}=\rho \,\vp{A_G}\in\R^m$,
where $\rho\neq0$ and $m=p(p-1)/2$.  We enumerate entries of $\vp{K^\ast}$ via ordered pairs $e=(i,j)$, $i>j$. Thus, $x^\ast_{e} = \rho$ if and only if $e\in E$. 
Let $d:=\max_{i\in V}\deg_G(i)+1$ denote one plus the maximum vertex degree.

We investigate three regularizers $\poly{\cdot}$: the $\ell_1$ norm, the $\ell_{\infty}$ norm, and the SLOPE norm $\|\cdot\|_w$ with weight vector $w$ (see Example \ref{ex:SLOPE}). 
For the $\ell_1$ norm we use $\cnorm{\cdot}=\|\cdot\|_\infty$, while for $\ell_\infty$ and SLOPE norms we use the projection-adapted norm $\|\cdot\|_{\Q}$ defined in \eqref{eq:projection-adapted-norm}.  
To ease the notation, we will use the norm subscripts $1$, $\infty$ and $w$ to indicate objects corresponding to a given norm.

The $\ell_1$-pattern $\mathrm{patt}_1$ can be identified with the $\mathrm{sign}$ function, while the $\ell_\infty$-pattern corresponds to the signs of the entries that attain the maximum absolute value, see Examples \ref{ex:l1} and \ref{ex:linfty}. For strictly decreasing positive weights, the SLOPE pattern of $x\in\R^m$ can be uniquely represented by the vector $M_x\in\R^m$ defined by 
$(M_x)_i = \mathrm{sign}(x_i )\mathrm{rank}(|x|)_i$ for $i=1,\ldots,m$, see \cite[Definition 2.1]{7aut}.

 The corresponding pattern subspaces for $x^{\ast}$ are
\[
\begin{array}{@{}l@{\quad}l@{}}
  \begin{aligned}
    \SS_1 &= \{\,x\in \R^m\colon x_e=0 \text{ for } e\notin E\},\\
    \SS_\infty &= \{\,x\in \R^m\colon x_e = x_{e'} \text{ for all }e,e'\in E\},
  \end{aligned}
  &
  \vcenter{\hbox{\qquad $\displaystyle \SS_w \;=\;\SS_1\cap \SS_\infty$.}}
\end{array}
\]

Denote $k = |E|$. The corresponding face projections are
\begin{align*}
f_1 = \mathrm{sign}(x^\ast) = \mathrm{sign}(\rho)\vp{A_G},\qquad f_\infty = f_1 / \|f_1\|_1, \qquad  f_{w}=w_{(1,k)} f_1, 
\end{align*}
where $w_{(i,j)}$ denotes the average $\frac{1}{j-i+1} \sum_{\ell=i}^j w_\ell$. Note that $\|f_1\|_1 =k$ and  Lemma \ref{lem:SLOPEtuning} gives
\[
\tau_1=1,\quad \tau_\infty = k^{-1}, \quad 
\tau_w = \min\left\{ h(w_1,\ldots,w_k), w_{(k+1,m)}\right\},
\]
with $h$ given explicitly. 

For the chosen auxiliary norms, we use
\[
c_1=1,\quad\!\!
c_\infty=
\max\left\{
    \|D f_\infty\|_{\Q},\,
    \max_{e\notin E}\|D\mathbf e_e\|_{\Q}
\right\},
\quad\!
c_w=w_{(1,k)}\|D f_1\|_{\Q},
\]
where $\mathbf e_e\in\R^m$ has entry one at the coordinate
indexed by $e=(i,j)$ and zero elsewhere, and each $\Q$
corresponds to the relevant penalty.
Indeed, for the $\ell_\infty$ penalty, the restricted dual unit ball
is the convex hull of $\pm f_\infty$ and
$\pm\mathbf e_e$, $e\notin E$, which gives the stated
expression for $c_\infty$. For the SLOPE, the restricted dual
unit ball is the line segment with endpoints
$\pm f_w=\pm w_{(1,k)}f_1$, which gives the stated
expression for $c_w$.
Let
\[
\cH_0=
\min\left\{
d\,\opnorm{\Sigma^\ast}_\infty
   \opnorm{K^\ast}_\infty^2,\,
\|\v{K^\ast}\|_1
\right\},
\]
and, for each penalty, write
$L_\diamond=
\max\left\{
\opnorm{\Q}_\infty,\,
\opnorm{I_{p^2}-\Q}_\infty
\right\}$.
The comparison constants are
\[
\cH_1=\cH_0,\qquad
\cH_\infty=2L_\infty\|\v{K^\ast}\|_1,\qquad
\cH_w=2L_w\cH_0.
\]
The stated choices follow from \eqref{eq:GDQ}, the norm comparison
$\|v\|_\infty\leq2\|v\|_{\Q}\leq2L_\diamond\|v\|_\infty$,
and the $\ell_\infty$ bounds in the proof of
Theorem~\ref{thm:glasso} and Lemma~\ref{lem:kappainfty}.

If $x\in \SS_1$ and $\mathrm{patt}_1(x)\neq\mathrm{patt}_1(x^\ast)$, then there exists $e\in E$ such that $\mathrm{sign}(x_e)\neq \mathrm{sign}(x_e^\ast)=\mathrm{sign}(\rho)$. 
If $x\in \SS_\infty$ and
$\mathrm{patt}_\infty(x)\neq\mathrm{patt}_\infty(x^\ast)$,
then there exist $e\notin E$ and $e'\in E$ such that
$|x_e|\geq |x_{e'}|$, or there exists $e\in E$ with
$\mathrm{sign}(x_e)\neq\mathrm{sign}(\rho)$.
If $x\in \SS_w$ and $\mathrm{patt}_w(x)\neq\mathrm{patt}_w(x^\ast)$, then $x = \gamma\, x^\ast$ for some $\gamma\leq 0$. In each case, by the second part of Lemma~\ref{lem:zeta}
and the equality of the pattern-stability thresholds for
$\|\cdot\|_\infty$ and $\|\cdot\|_{\Q}$, we obtain
$\zeta_\diamond(K^\ast) \geq  0.5 \opnorm{K^\ast}^{-2}_\infty |\rho|$ (for the $\ell_1$ and SLOPE penalties, the factor $0.5$
can be replaced by $1$.)

We consider four graph families $G=(\{1,\ldots,p\},E)$. In the descriptions below, conditions on off-diagonal entries are stated for $i>j$ and extended by symmetry.
\begin{itemize}
    \item Chain: $\Sigma^{\ast}_{ii}=1$, $\Sigma^{\ast}_{ij} = 0.2$ for $(i,j)\in E$ and $K^\ast_{ij}=0$ for all remaining off-diagonal entries; here $\rho=-5/24$,
    \item Hub (star): $\Sigma^\ast_{ii}=1$, $\Sigma^{\ast}_{ij} = 2.5/(p-1)$ for $(i,j)\in E$ and $K^\ast_{ij}=0$ for all remaining off-diagonal entries; here $\rho=-2.5(p-1)/((p-1)^2-2.5^2)$,
    \item Four-nearest neighbor grid: $K_{ii}^\ast=1$ and $K^\ast_{ij}=0.1=\rho$ for $(i,j)\in E$ and $K^\ast_{ij}=0$ for all remaining off-diagonal entries,
    \item Dense: $K^\ast_{ii}=1$, $K^\ast_{ij}=0.1=\rho$ if $\max\{i,j\}<p$ and $i\neq j$  and $K^\ast_{ij}=0$ for all remaining off-diagonal entries. 
\end{itemize}

\subsection{Irrepresentability condition}\label{sec:irrepcond}
Since the irrepresentability condition is identical for
$\cnorm{\cdot}=\|\cdot\|_\infty$ and
$\cnorm{\cdot}=\|\cdot\|_{\Q}$,
condition~\eqref{eq:irrep0} becomes, (respectively for $\poly{\cdot}=\|\cdot\|_1, \|\cdot\|_\infty, \|\cdot\|_w$)
\begin{gather*}
	\| (I_{p^2}-\PP_1) \Gamma^\ast  (\PP_1\Gamma^\ast \PP_1)^+D f_1\|_\infty \leq 1-\alpha, \\
	\| (I_{p^2}-\PP_\infty) \Gamma^\ast  (\PP_\infty \Gamma^\ast \PP_\infty)^+  Df_1 \|_\infty \leq 1-\alpha,\\
	\| (I_{p^2}-\PP_1 \PP_\infty) \Gamma^\ast  (\PP_\infty\PP_1 \Gamma^\ast \PP_1\PP_\infty)^+ D f_1 \|_\infty \leq  (1-\alpha) \frac{\tau_w}{w_{(1,k)}},
\end{gather*}
where $\PP_1$ and $\PP_\infty$ are the orthogonal projections onto $\mathcal{M}_1 $ and $\mathcal{M}_\infty$, respectively, and ${\tau_w}$ and 
$w_{(1,k)}$ are as defined at the beginning of this section. 

The term on the left-hand side of the third inequality is independent of the weight sequence $w$.  For the OSCAR weights (linearly decreasing from $1$ to $1/(2m-1)$, \citealp{OSCAR08}; Lemma~\ref{lem:SLOPEtuning} gives
\[
\frac{\tau_w}{w_{(1,k)}}
=
\frac{\min\{\lceil k/2\rceil,m-k\}}{2m-k}.
\]
For the graph families considered here, this ratio decreases at rate $p^{-1}$, which is very stringent. However, we may maximize  the right hand side so that the irrepresentability condition is less restrictive. E.g. adopting instead the tuned weights (see Lemma \ref{lem:SLOPEtuning}), $w_i=1$ if $i \leq \lfloor k/2\rfloor$, $w_i=2/3$ if $i = \lceil k/2\rceil$ and $k$ is odd and $w_i=1/3$ otherwise. This leads to $
\tau_w/w_{(1,k)}=(1/3)/(2/3)=1/2$.

Table~\ref{tab:graph_comparison_simple} reports the left-hand sides above for $p\in\{16,64\}$ and the four graph topologies.  The irrepresentability condition is satisfied for all graph
families and penalties considered.

\begin{table}[htbp]
	\centering
    	\caption{Irrepresentability conditions for different graph types and p values.}
	\begin{tabular}{l cccccc} 
		\toprule
$\| (I_{p^2}-\PP) \Gamma^\ast  (\PP \Gamma^\ast \PP)^+  D f_1 \|_\infty$		& \multicolumn{3}{c}{$p = 16$} & \multicolumn{3}{c}{$p = 64$} \\
		\cmidrule(lr){2-4} \cmidrule(lr){5-7} 
		Graph Type       & $1$ & $\infty$ & $w$ & $1$ & $\infty$ & $w$ \\
		\midrule
		Chain graph      & 4e-01 & 7e-02 & 4e-01 & 4e-01 & 7e-02 & 4e-01 \\
		Hub graph        & 3e-01 & 0 & 3e-01 & 8e-02 & 0 & 8e-02 \\
		Grid             & 4e-01 & 2e-02 & 4e-01 &  4e-01 & 3e-02 & 4e-01 \\
		Dense graph & 0 & 0 & 0 & 0 & 0  &  0  \\
		\bottomrule
	\end{tabular}
	\label{tab:graph_comparison_simple}
\end{table}

\subsection{Comparison with Ravikumar et al. (2011)}

Table~\ref{tab:graph_comparison_simple2} contrasts  our bound $\delta$ from Theorem \ref{cor:main} (in the case $\poly{\cdot}=\|\cdot\|_1$ and $\cnorm{\cdot}=\|\cdot\|_{\infty}$) with the threshold $\delta_R$ from \citet{ravikumar2011graphical} (recall \eqref{eq:goodDelta}). Across all scenarios, our bound exceeds that of \citet{ravikumar2011graphical} by several orders of magnitude, yielding markedly stronger guarantees for pattern recovery.

\begin{table}[htbp]
	\centering
    \caption{Comparison of $\delta$ and $\delta_{R}$ for different graph types and dimensions.}
	\begin{tabular}{l cccc} 
		\toprule
		& \multicolumn{2}{c}{$p = 16$} & \multicolumn{2}{c}{$p = 64$} \\
		\cmidrule(lr){2-3} \cmidrule(lr){4-5} 
		Graph Type       & $\delta$ & $\delta_{R}$ & $\delta$ & $\delta_{R}$ \\
		\midrule
		Chain graph      & 1.9e-03 & 1.6e-05 & 1.9e-03 & 1.6e-05 \\
		Hub graph        & 1.1e-03 & 2.3e-08 & 6.3e-04 & 1.6e-08 \\
		Grid             & 1.2e-03 & 1.0e-05 & 1.2e-03 & 9.4e-06 \\
		Dense graph & 1.4e-03 & 8.2e-07 & 1.1e-04 & 1.4e-09 \\
		\bottomrule
	\end{tabular}
\label{tab:graph_comparison_simple2}
\end{table}

\subsection{\texorpdfstring{Deterministic tightness of the threshold $\delta$}{Deterministic tightness of the threshold delta}}

The purpose of this experiment is to assess the sharpness of the deterministic recovery threshold in Theorem~\ref{cor:main}, rather than to study recovery as a function of the sample size. The theorem conditions directly on the realized covariance deviation \(\cnorm{\v{\hat\Sigma-\Sigma^\ast}}\). We therefore perturb \(\Sigma^\ast\) directly and compare the resulting recovery behavior with the theoretical threshold \(\delta\). This design isolates the deterministic recovery mechanism from the distribution-specific concentration behavior of a sample covariance matrix. 

As in the previous subsections, we use $\cnorm{\cdot}=\|\cdot\|_\infty$ for GLASSO and
$\cnorm{\cdot}=\|\cdot\|_{\Q}$ for the $\ell_\infty$ and tuned
SLOPE penalties. All experiments in this subsection use $p=25$.
For each graph and penalty, draw independent
$H_{ij}\sim\mathrm{Unif}(-1,1)$ for $i\leq j$, set
$H_{ji}=H_{ij}$, and independently draw $U\sim\mathrm{Unif}(0,4\delta)$.
Set $\mathcal E=UH/\cnorm{\v H}$ and
$\hat\Sigma=\Sigma^\ast+\mathcal E$. Then, $\cnorm{\v{\mathcal E}}\sim\mathrm{Unif}(0,4\delta)$. For every configuration used in these experiments, the $\hat\Sigma$ matrices are always positive definite. Using the optimal tuning parameter $\lambda$ from
Theorem~\ref{cor:main}, we compute
the estimator $\hat K$ with
$\poly{\cdot}\in\{\|\cdot\|_1,\|\cdot\|_\infty,\|\cdot\|_w\}$.
Repeating this experiment $N=1\,000$ times, we record $
\left(
\cnorm{\v{\hat\Sigma_i-\Sigma^\ast}},
\|\v{\hat K_i-K^\ast}\|_\infty
\right)$
and display the results as scatter plots. We compare the theoretical $\delta$ with the empirical threshold
$\hat{\delta} = \min\{ \cnorm{\v{\hat{\Sigma}_i-\Sigma^{\ast}}}\colon \patt{\hat{K}_i}\neq\patt{K^\ast}\}$
and report the threshold ratio $\hat\delta/\delta$.
Thus, $\hat\delta$ is the smallest tested deviation at which recovery
failed.

In the case of GLASSO, we find that $\hat{\delta}$ exceeds the theoretical bound $\delta$ from Theorem \ref{cor:main} by only a factor of about $2$, see Figure \ref{fig:glasso}.
\begin{figure}[htbp]\centering
	\includegraphics[width=0.8\textwidth]{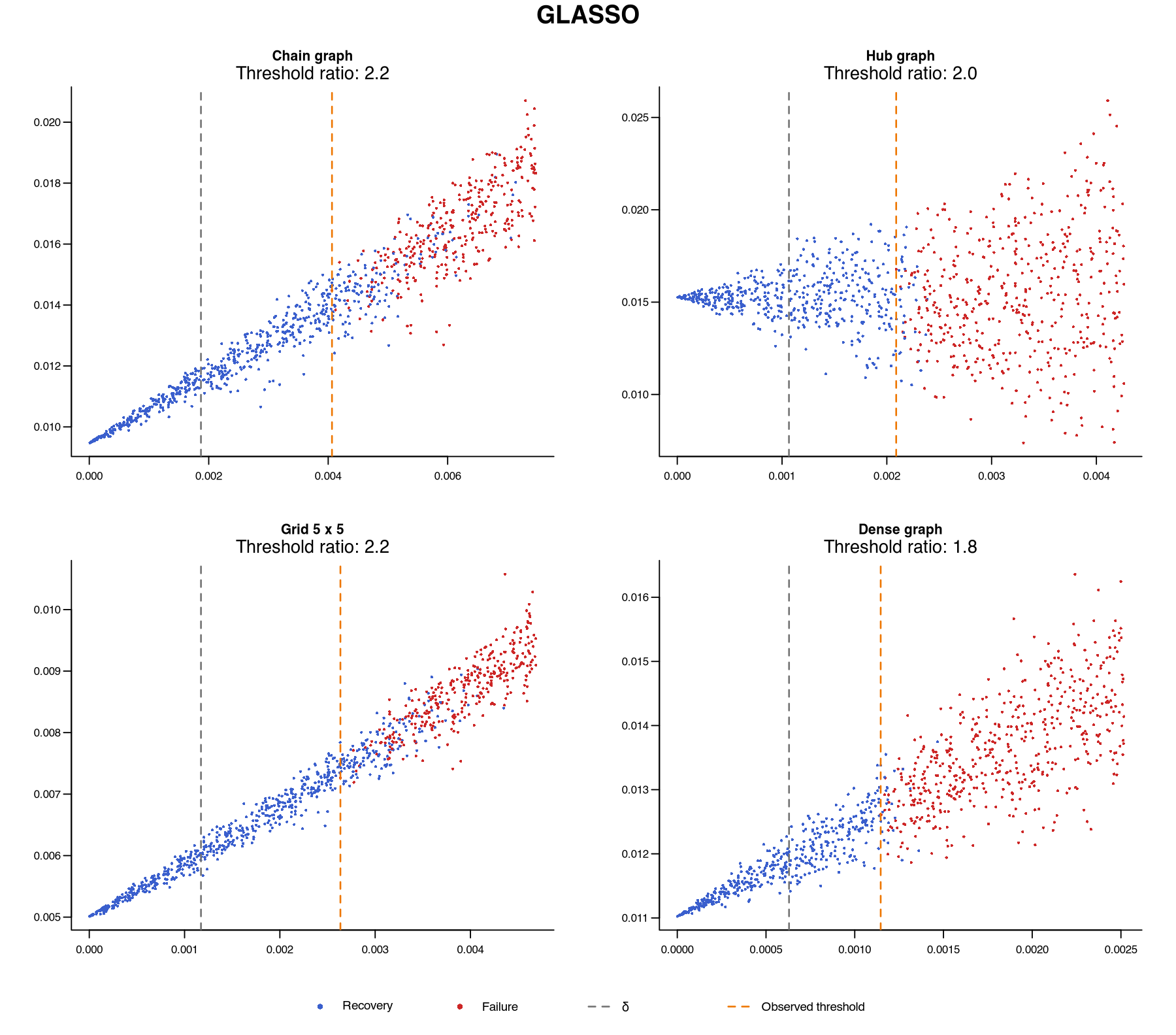}
\caption{GLASSO: Scatter plot of \(\|\v{\hat{\Sigma}-\Sigma^{\ast}}\|_{\infty}\) (x-axis) versus 
\(\|\v{\hat{K}-K^{\ast}}\|_{\infty}\) (y-axis) for $p=25$.  
The gray dashed line marks the theoretical threshold \(\delta\); the orange dashed line shows its empirical threshold \(\hat{\delta}\).  
Points are colored blue when the true pattern (sign) is recovered and red otherwise.}\label{fig:glasso}

\end{figure}

We carried out an analogous evaluation for the SLOPE penalty,
using the weights specified in Section~\ref{sec:irrepcond}
and the projection-adapted auxiliary norm $\|\cdot\|_{\Q}$.
The results are displayed in Figure~\ref{fig:slope}.
The theoretical thresholds are also within a factor of about $2$ of the observed failure thresholds. 
The $\ell_\infty$ penalty with the projection-adapted auxiliary
norm produced similar results, with threshold ratios close to $2$
for all four graph families (not shown; the figure can be generated using the reproduction code). 

\begin{figure}[htbp]
\centering
\includegraphics[width=0.8\textwidth]{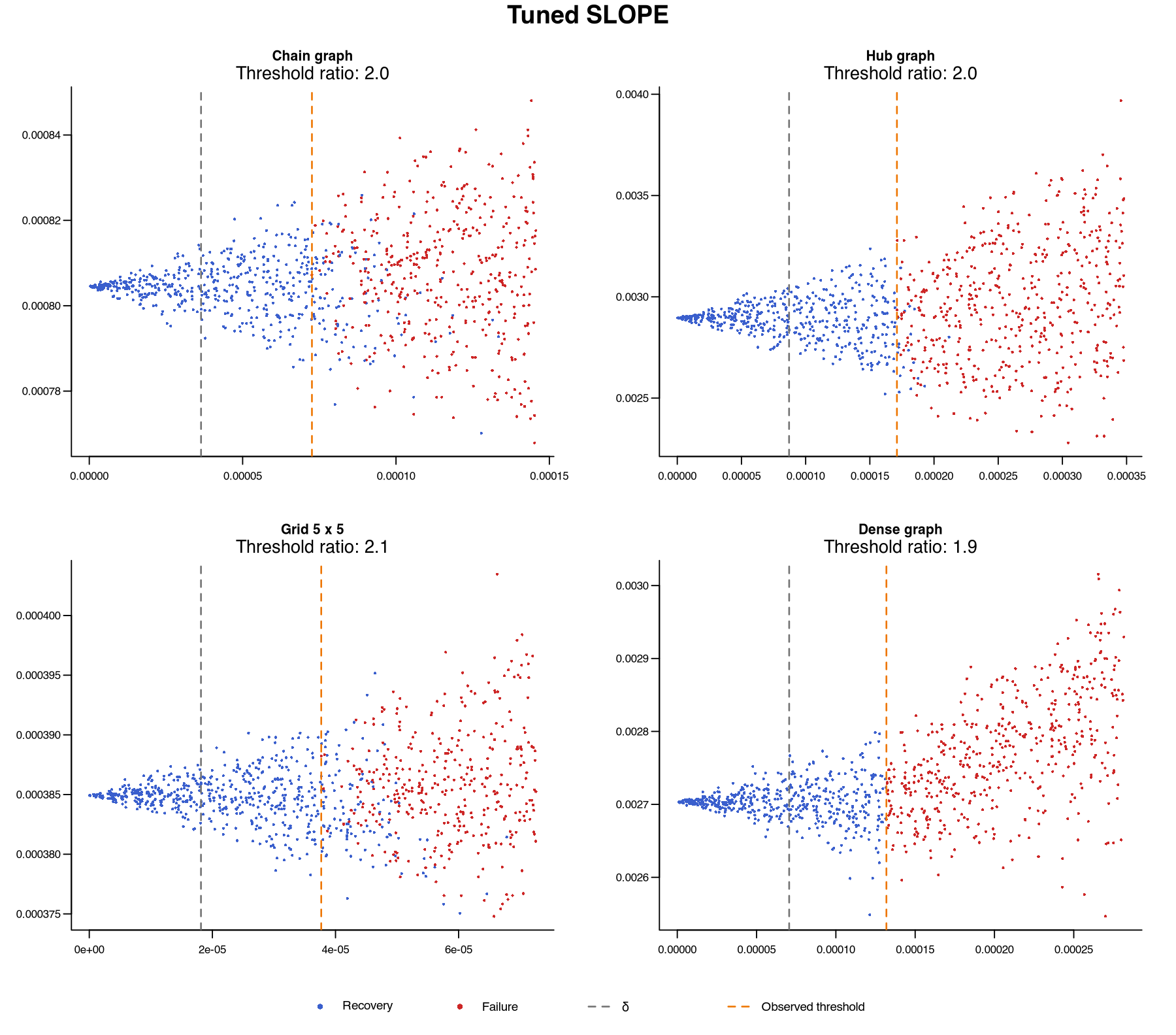}
\caption{SLOPE: Scatter plot of \(\|\v{\hat{\Sigma}-\Sigma^{\ast}}\|_{\Q}\) (x-axis) versus 
\(\|\v{\hat{K}-K^{\ast}}\|_{\infty}\) (y-axis) for $p=25$.  
The gray dashed line marks the theoretical threshold \(\delta\); the orange dashed line shows its empirical threshold \(\hat{\delta}\).  
Points are colored blue when the true pattern is recovered and red otherwise. }\label{fig:slope}
\end{figure}

For comparison, we also considered the $\ell_\infty$ and SLOPE penalties with
$\cnorm{\cdot}=\|\cdot\|_\infty$.
For both penalties and all four graph families, the observed threshold ratio exceeded $4$ (not shown here). The larger threshold ratios may reflect a mismatch between the
$\ell_\infty$ norm and the geometry of the target patterns.

We thank Prof. Ma{\l}gorzata Bogdan for sharing the GSLOPE implementation associated with \cite{graphSLOPE}.
The reproduction code is available at
\url{https://github.com/kolodziejek/pattern-recovery}.
OpenAI's Codex was used to assist with the preparation of this code.

\section{Roadmap to the Proofs}\label{sec:ProofSkeleton}

We will prove a result more general than Theorem \ref{cor:main}. Theorem \ref{cor:main} will then follow as a corollary.
\begin{theorem}\label{thm:patt_recov}
Suppose that for some $(r,\lambda)\in(0,1)\times (0,\infty)$, we have $r\leq \cH\zeta_\diamond(K^\ast)$ and
	\begin{align}\label{eq:smallr}
		\cnorm{\v{\hat{\Sigma}-\Sigma^\ast}}
		<
		\frac{r}{\cH\copnorm{\Q}}
		-\lambda\, c_\diamond
		-\frac1{\cH}\frac{r^2}{1-r}.
	\end{align}
If $|\I^\ast|\geq2$, suppose additionally that there exists $\alpha\in(0,1)$ for which \eqref{eq:irrep0} holds and
	\begin{align}\label{eq:smallrdual}
		\cnorm{\v{\hat{\Sigma}-\Sigma^\ast}}
		\leq
		\frac{\alpha\lambda\tau_\diamond(\I^\ast)}
		{\copnorm{I_{p^2}-\Q}}
		-\frac1{\cH}\frac{r^2}{1-r}.
	\end{align}
	Let $\hat{K}$ be the unique solution to \eqref{eq:hatK}. Then,  $\patt{\hat{K}}=\patt{K^\ast}$ and 
	\[
	\cnorm{\Gamma^\ast\v{\hat{K} - K^\ast}}\leq {r}/{\cH}.
	\]
\end{theorem}

We give a general outline of the primal–dual witness argument, the key residual-control lemmas, and the Brouwer‐fixed‐point construction that yields our main theorems. The proofs of presented lemmas can be found in the Appendix \ref{app:proofs}.

\subsection{Primal-dual witness method}\label{sec:pdwm}
We consider the original and the restricted problem:
\begin{align}\label{eq:hatK2}
	\hat K &= \argmin_{K\in \mathrm{Sym}_+(p)} \left\{ \tr{\hat{\Sigma}K}-\log\det(K)+\lambda \norm(K)\right\},\\
	\tilde K &=\argmin_{K\in \Mast\cap \mathrm{Sym}_+(p)} \left\{ \tr{\hat{\Sigma}K}-\log\det(K)+\lambda \norm(K)\right\}. \label{eq:tildeK}
\end{align}
Recall that $\norm$ is an atomic norm on $\mathrm{Sym}_0(p)$
(see Section \ref{sec:vectorization} for the definition of $\mathrm{Sym}_0(p)$). When $|\I^\ast|=1$, we have $\Mast=\mathrm{Sym}(p)$, so the optimization problems \eqref{eq:hatK2} and \eqref{eq:tildeK} coincide, Remark \ref{rem:singleton0}.
We define the corresponding restricted dual norm on $\Mast_0= \Mast\cap \mathrm{Sym}_0(p)$. Let $\norms = \norm\big|_{\Mast}$ and 
\[
\normsd(Y) = \sup\{ \scalar{Y}{X}\colon X\in\Mast_0, \norms(X)\leq 1\},\quad Y\in \Mast_0.
\]
It follows that $\normsd(Y) = \polydr{\vp{Y}}{\I^\ast}$ with $\I^\ast =\I_{\vp{K^\ast}}$. 

Let $\proj\colon \mathrm{Sym}(p)\to \Mast$ denote the orthogonal projection onto $\Mast$.
Define also the orthogonal projection matrix $\PM \in\mathrm{Sym}(p^2)$ such that for all $X\in\mathrm{Sym}(p)$, the following holds: 
\begin{align*}
	\v{\proj(X)}=\PM \v{X}.
\end{align*}
Note that we have $\vp{\proj(X)}=\PP_{\I^\ast} \vp{X}$ (recall Section \ref{sec:pattern}).

\begin{lemma}\label{lem:1}
	Assume that $\hat{\Sigma}\in\mathrm{Sym}_{\geq 0}(p)$ has positive diagonal entries.
	\begin{itemize}
		\item[(i)] The optimization problem \eqref{eq:hatK2} has a unique solution $\hat{K}\in\mathrm{Sym}_+(p)$, which is characterized by
		$\hat{K}^{-1}-\hat{\Sigma} = \lambda\hat{\Pi}$, 
		where $\hat{\Pi}\in \partial \norm(\hat{K})$.
		\item[(ii)] The constrained optimization problem \eqref{eq:tildeK} has a unique solution 
		$\tilde{K}\in \Mast\cap\mathrm{Sym}_+(p)$, which is  characterized by $\proj(\tilde{K}^{-1}-\hat{\Sigma}) = \lambda\tilde\Pi$,
		where $\tilde\Pi\in \partial \norms(\tilde{K})$.
	\end{itemize}
\end{lemma}

To establish that the original estimator $\hat{K}$ belongs to the true model subspace $\Mast$, it suffices to prove that $\hat{K}=\tilde{K}$. In Lemma \ref{lem:1}, we define
\[
\hat{\Pi} = \frac{1}{\lambda}(\hat{K}^{-1}-\hat{\Sigma})\quad\mbox{and}\quad \tilde{\Pi} = \frac{1}{\lambda}\proj(\tilde{K}^{-1}-\hat{\Sigma}). 
\]
Let additionally
\begin{align*}
	\Pi = \frac{1}{\lambda}(\tilde{K}^{-1}-\hat{\Sigma}).
\end{align*} 
The next result establishes a condition under which $\Pi=\hat{\Pi}$; this, in turn, implies that $\hat{K}=\tilde{K}$, thereby ensuring that $\hat{K}$ belongs to the true pattern subspace $\Mast$.

\begin{lemma}\label{lem:hatKtildeK}
	We have $\normd(\Pi)\leq 1$ if and only if $\hat{K}=\tilde{K}$. 
\end{lemma}

\subsection{The basic bounds}\label{sec:bb}
Next, we find a sufficient condition under which we have $\normd(\Pi)\leq 1$. This condition will be formulated in terms of a projection $\Q$, recall its definition from Section \ref{sec:notation}.

For $\Delta\in \Mast$ such that $K^\ast+\Delta$ is invertible,  define the residual matrix
\[
R(\Delta) = (K^\ast+\Delta)^{-1}-\Sigma^\ast+\Sigma^\ast\Delta\Sigma^\ast. 
\]
The matrix $R(\Delta)$ is the remainder in the first-order expansion
of $(K^\ast+\Delta)^{-1}$ at $K^\ast$.

Recall the definition of $\tau_\diamond$ in Definition \ref{def:tau}. We obtain
\begin{lemma}\label{lem:Vbound}
	Assume that $\patt{\tilde{K}}=\patt{K^\ast}$.

	If either $|\I^\ast|=1$, or $|\I^\ast|\geq2$ and there exists
	$\alpha\in(0,1)$ such that \eqref{eq:irrep0} holds and
	\[
	\cnorm{\v{R(\tilde{K}-K^\ast)}}
	+
	\cnorm{\v{\hat{\Sigma}-\Sigma^\ast}}
	\leq
	\frac{\alpha\,\lambda\,\tau_\diamond(\I^\ast)}
	{\copnorm{I_{p^2}-\Q}},
	\]
	then $\normd(\Pi)\leq1$.
\end{lemma}

While the term $\cnorm{\v{\hat{\Sigma}-\Sigma^\ast}}$,      
is stochastic and cannot be directly controlled, standard concentration inequalities ensure it can be made arbitrarily small with high probability by increasing the sample size $n$. The term $\cnorm{ \v{R(\tilde{K}-K^\ast)}}$ will be handled using the following result.

\begin{lemma}[Control of the residual term]\label{lem:Rbound}
	If $\Delta\in\Mast$ satisfies 
    \[
	\cH\cnorm{\Gamma^\ast\v{\Delta}}\leq r <1,\qquad\mbox{then}\qquad 
		\cnorm{\v{R(\Delta)}} \leq  \frac{1}{\cH}\frac{r^2}{1-r}.
	\]
\end{lemma}

\subsection{Fixed-point argument}\label{sec:fpa}
To obtain control over the difference $\tilde{\Delta}=\tilde{K}-K^\ast$, we employ a similar approach as in \cite{ravikumar2011graphical}. 
In this section,
we define a continuous mapping $F$ that admits the unique fixed point
\[
\Gamma^\ast \v{\tilde{K}-K^\ast}. 
\]
Next, we will show that $F$ is a self-map on a particular convex set. Consequently, Brouwer's fixed point theorem guarantees that this fixed point lies within that convex set.

Define the function $F\colon  \{ \Gamma^\ast\v{\Delta}\colon \Delta\in\Mast, \,\opnorm{\Sigma^\ast\Delta}<1 \} \to \Gamma^\ast\v{\Mast}$ by
\begin{align}\label{eq:defF}
F\left(\Gamma^\ast\,\v{\Delta}\right) = \Q\,\v{(K^\ast+\Delta)^{-1}-\tilde{K}^{-1}} + \Gamma^\ast\,\v{\Delta}.
\end{align}

We have $K^\ast+\Delta=K^\ast(I_p+ \Sigma^\ast\Delta)$ and since $\opnorm{\Sigma^\ast\Delta}<1$, the matrix $(I_p+ \Sigma^\ast\Delta)$ is invertible. Hence, $K^\ast+\Delta$ remains positive definite, ensuring that $F$ is well-defined. Here, $\opnorm{\cdot}$ can be arbitrary operator norm. 

\begin{lemma}\label{lem:Ffixed}
	$\Gamma^\ast\v{\tilde{\Delta}}$ is a fixed point of $F$ if and only if $\tilde\Delta = \tilde{K} - K^\ast$.
\end{lemma}

\begin{lemma}[Control of the deviation $\tilde{\Delta}$]\label{lem:Deltabound}
Suppose that $r\in(0,1)$, $\lambda>0$, and
	condition~\eqref{eq:smallr} holds.
	Then, $F$ has a unique fixed point $\Gamma^\ast\v{\tilde{K}-K^\ast}$, $\tilde{K}-K^\ast\in\Mast$, which satisfies 
	\begin{align*}
		\cnorm{\Gamma^\ast\v{\tilde{K}-K^\ast}}< 	\frac{r}{\cH}.
	\end{align*}
	If additionally, $r \leq \cH \zeta_{\diamond}(K^\ast)$,
	then $\patt{\tilde{K}}=\patt{K^\ast}$. 
\end{lemma}

From Lemma \ref{lem:Deltabound}, we obtain the following immediate result
\begin{corollary}\label{cor:bound}
	Assume that  
$x := \cH(\cnorm{\v{\hat{\Sigma}-\Sigma^\ast}}  + \lambda\,c_\diamond) \leq  \frac{0.1}{\copnorm{\Q}^2}$.
	Then, $r=\copnorm{\Q} x+ 2 \copnorm{\Q}^3 x^2$ satisfies \eqref{eq:smallr} and 
	\[
	\cnorm{\Gamma^\ast\v{\tilde{K}-K^\ast}} \leq  \frac{r}{\cH}\leq 1.2\copnorm{\Q} (\cnorm{\v{\hat{\Sigma}-\Sigma^\ast}}  + \lambda\,c_\diamond).
	\]
	
\end{corollary}

\subsection{Proofs of the main results}\label{sec:mr}
\begin{proof}[Proof of Theorem \ref{thm:patt_recov}]
	Under \eqref{eq:smallr}, by Lemma \ref{lem:Deltabound}, $F$ has a fixed point $\Gamma^\ast\v{\tilde{\Delta}}\in\Gamma^\ast\v{\Mast}$ such that
	\[
	\cH\cnorm{\Gamma^\ast\v{\tilde{\Delta}}}< r,
	\]
	where $\tilde\Delta = \tilde{K}-K^\ast$. Since $r\leq \cH\,\zeta_\diamond(K^\ast)$, we have $\patt{\tilde{K}}=\patt{K^\ast}$.

	By Lemma \ref{lem:Rbound}, the residual term $R=R(\tilde{\Delta})$ satisfies 
	\begin{align*}
		\cnorm{ \v{R}} &\leq \frac{1}{\cH}\frac{r^2}{1-r}. 
	\end{align*}

    If $|\I^\ast|\geq2$, condition~\eqref{eq:smallrdual} and the
	residual bound give
	\begin{align*}
		\cnorm{\v R}
		+
		\cnorm{\v{\hat\Sigma-\Sigma^\ast}}
		\leq
		\frac1{\cH}\frac{r^2}{1-r}
		+
		\cnorm{\v{\hat\Sigma-\Sigma^\ast}} \leq
		\frac{\alpha\lambda\tau_\diamond(\I^\ast)}
		{\copnorm{I_{p^2}-\Q}}.
	\end{align*}
	Thus, in both cases, Lemma~\ref{lem:Vbound} gives
	$\normd(\Pi)\leq1$. Lemma~\ref{lem:hatKtildeK} then implies
	$\hat K=\tilde K$, which completes the proof.
\end{proof}

\begin{lemma}\label{lem:optim}
	Assume \(z\in(0,1]\) and define
	\[
	\delta
	=\sup_{\substack{0<r<1,\ r\leq z\\\lambda>0}}\left\{ \min\left\{\frac{r}{\cH\copnorm{\Q}} - \lambda\,c, \lambda\,\alpha\tau'
	\right\} - \frac1{\cH}\frac{r^2}{1-r}\right\}.
	\]
	Then the supremum is attained at the unique pair
	\[
	r_o=\min\{r^*,z\},
	\qquad
	\lambda_o=\frac{r_o }{\cH\copnorm{\Q}(  c+\alpha\tau')},
	\]
	where $r^\ast=1-\frac{1}{\sqrt{1+M}}$ and $M = \frac{ \,\alpha\tau'}{\copnorm{\Q}(  c+\alpha\tau')}$.
	Moreover,
	\[
	\delta=\frac{1}{\cH}
	\begin{cases}
		\bigl(\sqrt{1+M}-1\bigr)^2,
		&\mbox{if }r^\ast\leq z,\\
	M z
		-\tfrac{z^2}{1-z},
		&\mbox{if }r^\ast>z.
	\end{cases}
	\]
\end{lemma}

\begin{proof}
	First fix $r$ and consider the objective function as a function of $\lambda$. Clearly, the supremum in $\lambda$ is attained when 
	\[
	\frac{r}{\cH\copnorm{\Q}} - \lambda\,c =  \alpha\,\lambda\,\tau',\qquad \mbox{i.e.,}\qquad
	\lambda=\frac{r}{\cH\copnorm{\Q}(  c+\alpha\tau')}.
	\]
	Substituting back gives a one-dimensional problem
	\[
	\delta= \frac{1}{\cH}\sup_{0<r<1,\ r\leq z}\Bigl\{M r - \frac{r^2}{1-r}\Bigr\}.
	\]
	A short calculation shows that the optimum is attained at  
$r_o=\min\{ r^\ast, z\}$. Inserting \(r_o\) into the objective produces the two cases in the statement.
\end{proof}

\begin{proof}[Proof of Theorem \ref{cor:main}]
	The proof follows from Theorem \ref{thm:patt_recov} by optimizing
    \[
	\delta
	=\sup_{\substack{0<r<1,\ r\leq z\\\lambda>0}}\left\{ \min\left\{\frac{r}{\cH\copnorm{\Q}} - \lambda\,c, \lambda\,\alpha\tau'
	\right\} - \frac1{\cH}\frac{r^2}{1-r}\right\}
	\]
    with $z=\min\{1, \cH\zeta_\diamond(K^\ast)\}$, $c=c_\diamond$, and $\tau' = \tau_\diamond/\copnorm{I_{p^2}-\Q}$.
    By Lemma \ref{lem:optim}, the optimizer is $r_o=\min\{r^\ast,z\}$, where $r^\ast=1-\frac{1}{\sqrt{1+M}}$.
    Since $r^\ast<1$, the case $r^\ast>z$ can occur only when $z=\cH\zeta_\diamond(K^\ast)$.
    Substituting the optimizer into Lemma \ref{lem:optim} yields the displayed formula for $\delta$ and the stated choice of $\lambda$.
    The conclusion then follows from Theorem \ref{thm:patt_recov} with $r=r_o$.
\end{proof}

\subsection{Proof of Theorem \ref{thm:glasso}}

\begin{lemma}\label{lem:kappainfty}
For any $\Delta\in\mathrm{Sym}(p)$, we have
\[
\opnorm{\Sigma^\ast \Delta}_\infty \leq \|\v{K^\ast}\|_1\cdot \|\Gamma^\ast\v{\Delta}\|_\infty.
\]
\end{lemma}
\begin{proof}
Set $\Delta= K^\ast X K^\ast$, where $0\neq X\in\mathrm{Sym}(p)$. We have with $X' = X/\|\v{X}\|_\infty$, 
\begin{align*}
\frac{ \opnorm{\Sigma^\ast \Delta}_\infty}{\|\Gamma^\ast\v{\Delta}\|_\infty} &= \frac{\opnorm{X K^\ast}_\infty}{\|\v{X}\|_\infty} = \opnorm{X' K^\ast}_\infty = \max_{i} \sum_{j} \left| \sum_k  X'_{ik} K^\ast_{kj} \right| \\
& \leq \max_{i} \sum_{j}  \sum_k \left| X'_{ik} K^\ast_{kj} \right|\leq \max_{i} \sum_{j}  \sum_k \left|K^\ast_{kj} \right| = \|\v{K^\ast}\|_1. 
\end{align*}
\end{proof}

\begin{proof}[Proof of Theorem \ref{thm:glasso}]
	Since $\poly{\cdot}$ is the $\ell_1$ norm, both $\polyd{\cdot}$ and $\polydr{\cdot}{\I^\ast}$ are $\ell_\infty$ norms. Let $\cnorm{\cdot}$ be the $\ell_\infty$-norm. In this setting we have  $c_{\diamond}=1$ and $\tau_\diamond(\I^\ast)=1$ provided $|\I^\ast|\geq 2$. Moreover, 
	\begin{align*}
		\zeta_\diamond(K^\ast) 
		&= 
		\inf_{\substack{K\in\Mast\\
				\patt{K}\neq\patt{K^\ast}}} \|\Gamma^\ast\v{K-K^\ast}\|_\infty\\
		&  \geq \opnorm{K^\ast}_\infty^{-2}\inf_{\substack{K\in\Mast\\
				\patt{K}\neq\patt{K^\ast}}} \|\v{K-K^\ast}\|_\infty = \opnorm{K^\ast}_\infty^{-2} \Theta_{\min},
	\end{align*}
	where we have denoted $\Theta_{\min} = \min_{(i,j)\in S} |K_{i,j}^\ast|$.
    By assumption, we obtain
	\[
	\|\Gamma_{2,1}^\ast(\Gamma_{1,1}^\ast)^+ \v{\mathrm{sign}(K^\ast)}\|_\infty \leq \opnorm{\Gamma_{2,1}^\ast (\Gamma_{1,1}^\ast)^+}_\infty = \opnorm{\Gamma_{S^c,S}^\ast(\Gamma_{S,S}^\ast)^{-1}}_\infty \leq 1-\alpha
	\]
	and therefore \eqref{eq:irrep0} is satisfied. 
    
	Since $\opnorm{I_{p^2}-\PM}_\infty=1$ and $\Gamma_{2,1}^\ast (\Gamma_{1,1}^\ast)^+=\Q-\PM$, we arrive at 
	\[
	\opnorm{I_{p^2}-\Q}_\infty\leq \opnorm{I_{p^2}-\PM}_\infty + \opnorm{\Q-\PM}_\infty\leq 2-\alpha.
	\]
	Moreover, $(\Q)_{\cdot,S^c}=0$ and $(\Q)_{S,S}=I_{|S|}$ imply that
	\begin{align*}
	\opnorm{\Q}_\infty=& \max\{ \opnorm{(\Q)_{S,S}}_\infty, \opnorm{(\Q)_{S^c,S}}_\infty\} \\
    &=  \max\{1,\opnorm{\Gamma_{S^c,S}(\Gamma_{S,S}^\ast)^{-1}}_\infty\}=1.
	\end{align*}
	Further, for $\Delta\in\Mast$ we have $\opnorm{\Delta}_\infty\leq d \|\v{\Delta}\|_\infty$. Therefore, 
	\begin{align*}
		\copnorm{\Sigma^\ast\Delta\otimes I_p} &= 	\opnorm{\Sigma^\ast\Delta}_\infty   \leq 
		\opnorm{\Sigma^\ast}_\infty \opnorm{\Delta}_\infty
		\leq d
		\opnorm{\Sigma^\ast}_\infty \|\v{\Delta}\|_\infty
		\\
		& \leq	d \opnorm{\Sigma^\ast}_\infty  \opnorm{K^\ast}^2 _\infty  \|\Gamma^\ast\v{\Delta}\|_\infty,
	\end{align*}
	where $d$ is the maximal number of non-zero entries in one row of $K^\ast$. In view of Lemma \ref{lem:kappainfty}, we obtain $\cH =\min\{ d \,\opnorm{\Sigma^\ast}_\infty  \opnorm{K^\ast}^2 _\infty, \|\v{K^\ast}\|_1\}$.  
	
	Set $r = \min\left\{ \tfrac{\alpha}{4}, \cH \opnorm{K^\ast}_\infty^{-2} \Theta_{\min}\right\}$ and $\lambda=\tfrac{r}{\cH}(1-\tfrac{\alpha}{2})$. Then, direct calculations show that
	\begin{align*}
		  \delta&=\frac{1}{\cH}\begin{cases}
 \frac{\alpha^2}{16}\left(1-\frac{\alpha}{3}\right), & \mbox{if }\alpha \leq 4 \cH \opnorm{K^\ast}_\infty^{-2} \Theta_{\min},\\
\frac{\alpha}{6} \cH \opnorm{K^\ast}_\infty^{-2} \Theta_{\min}, & \mbox{otherwise},
        \end{cases}
        \\
		&< \frac{1}{\cH}
        \begin{cases}
\frac{\alpha^2 (2-\alpha)}{8(4-\alpha)}, & \mbox{if }\alpha \leq 4 \cH \opnorm{K^\ast}_\infty^{-2} \Theta_{\min},\\
\frac{\alpha}{2}\cH \opnorm{K^\ast}_\infty^{-2} \Theta_{\min} - \frac{(\cH \opnorm{K^\ast}_\infty^{-2} \Theta_{\min})^2}{1- \cH \opnorm{K^\ast}_\infty^{-2} \Theta_{\min}}, & \mbox{otherwise,}
        \end{cases}\\
		&=\min\left\{ \frac{r}{\cH}-\lambda, \frac{\alpha\lambda}{2-\alpha}\right\}-\frac{1}{\cH}\frac{r^2}{1-r}. 
	\end{align*}
	Since $r\leq \cH \opnorm{K^\ast}_\infty^{-2} \Theta_{\min} \leq \cH \zeta_\diamond(K^\ast)$, by Theorem \ref{thm:patt_recov}, the condition 
	$\|\v{\hat{\Sigma}-\Sigma^\ast}\|_\infty \leq \delta$,
	implies that $\patt{\hat{K}}=\patt{K^\ast}$ (i.e., $\mathrm{sign}(\hat{K})=\mathrm{sign}(K^\ast)$) and
	\[
	\| \Gamma^\ast\v{\hat{K}-K^\ast}\|_\infty \leq \frac{r}{\cH} \leq \frac{\alpha}{4 \,\cH }.
	\]
	Finally, the bound $\| \v{\hat{K}-K^\ast}\|_\infty  \leq \opnorm{K^\ast}_\infty^2 \| \Gamma^\ast\v{\hat{K}-K^\ast}\|_\infty$,
	completes the proof.
\end{proof}

\section{Pattern recovery for skewed gauges}\label{app:pr}

  In relation to Example \ref{ex:simple}, we call ``skewed'' 
the gauges for which $\polyd{f_{\I^\ast}}>1$. Then, $\tau_\diamond$ does not exist and thus, the sufficient
condition of Theorem~\ref{cor:main} is unavailable. We show non-recovery when the
Hessian-induced off-pattern component is too small to compensate
for the position of $f_{\I^\ast}$.
The inequality $\polyd{f_{\I^\ast}}>1$ implies that for arbitrary norm $\|\cdot\|$ on $\SS_{\I^\ast}^\bot$, there exists positive $ a $ such that 
\[
\polydr{f_{\I^\ast}}{\I^\ast} <\min_{\pi^\bot\in \SS_{\I^\ast}^\bot\colon \|\pi^\bot\|\leq  a } \polyd{f_{\I^\ast} + \pi^\bot}.
\]
Take $\|\cdot\| = \cnorm{D \,\cdot}$, fixing some $ a >0$. 

We will show that the pattern recovery in the case when $\tau_\diamond(\I^\ast)$ does not exist will typically prohibit pattern recovery. 
This highlights that the existence of a positive threshold $\tau$ is crucial for pattern recovery. To show this, we assume a very simplified setting where $p$ is fixed, $n\to\infty$ and $\lambda=\lambda_n$ satisfies $\lambda_n\to 0$ together with $\sqrt{n}\lambda_n\to\infty$. To make our argument simpler, we assume a very strong irrepresentability condition which is (cf. Definition \ref{def:irrep})
\[
\cnorm{ (\Q-\PM) D f_{\I^\ast}} \leq\frac{ a }{2}. 
\]
Suppose 
\begin{align}\label{eq:x}
	x:= \cH(\cnorm{\v{\hat{\Sigma}-\Sigma^\ast}} + \lambda_n c_{\diamond} )< \frac{r}{\copnorm{\Q}}-\frac{r^2}{1-r} 
\end{align}
for some $r\in (0,\min\{\cH \,\zeta_\diamond(K^\ast),1\})$. We note that $x$ can be made arbitrarily small (with high probability) by taking sufficiently large $n$. 

Then, by Lemma \ref{lem:Deltabound}, $\patt{\tilde{K}}=\patt{K^\ast}$ (in particular, $\vp{\tilde{\Pi}}=f_{\I^\ast}$) and $\cnorm{\Gamma^\ast\v{\tilde{K}-K^\ast}}\leq r/\cH$. 

We will show that we have $\normd(\Pi)=\polyd{\vp{\Pi}}>1$, which, through Lemma \ref{lem:hatKtildeK},  will imply that $\hat{K}\neq\tilde{K}$. Since $\proj(\Pi) = \tilde{\Pi}$, we have 
\[
\vp{\Pi} = f_{\I^\ast}+\vp{\proj^\bot(\Pi)}. 
\]
If $\cnorm{\v{\proj^\bot(\Pi)}}\leq  a $, we obtain $\normd(\Pi)>1$. Indeed,
\[
1 = \polydr{f_{\I^\ast}}{\I^\ast} < \min_{\pi^\bot\in \SS_{\I^\ast}^\bot\colon \|\pi^\bot\|\leq  a }  \polyd{f_{\I^\ast} + \pi^\bot}\leq \polyd{\vp{\Pi}} = \normd(\Pi).
\]
Write $q=\copnorm{\Q}$ and $b=\copnorm{I_{p^2}-\Q}$.
For sufficiently small $x $, take $r=qx+2q^3x^2$ as in
Corollary~\ref{cor:bound}; its sufficient condition is $x\leq0.1/q^2$.
Assume $\E[\|X\|_2^4]<\infty$.
Since $p$ is fixed,
$\v{\hat\Sigma-\Sigma^\ast}=O_{\mathbb P}(n^{-1/2})$.
Together with $\sqrt n\lambda_n\to\infty$, this implies
$\cnorm{\v{\hat\Sigma-\Sigma^\ast}}/\lambda_n\to0$
in probability.
Thus $x/\lambda_n\to\cH c_\diamond$ and
$r/\lambda_n\to q\cH c_\diamond$ in probability,
so $r<\min\{1,\cH\zeta_\diamond(K^\ast)\}$ with probability tending to one.
The proofs of Lemmas~\ref{lem:Vbound} and~\ref{lem:Rbound} give
\[
\cnorm{\v{\proj^\bot(\Pi)}}
\leq \frac a2+\frac b{\lambda_n}
\left(\cnorm{\v{\hat\Sigma-\Sigma^\ast}}
      +\frac1{\cH}\frac{r^2}{1-r}\right)
=\frac a2+o_{\mathbb P}(1).
\]
Consequently, $\normd(\Pi)>1$ with probability tending to one.
Lemma~\ref{lem:hatKtildeK} then gives $\hat K\neq\tilde K$;
hence $\hat K$ does not recover the true pattern.

\section{Choice of tuning parameters in the SLOPE norm}\label{sec:SLOPE}
The dual SLOPE unit ball is a signed permutahedron, a well-studied object. Its geometry depends on the weights $w$ of the SLOPE norm, $\poly{\cdot}=\|\cdot\|_w$. In this section we present a quantitative approach which leads to improved pattern recovery properties of the SLOPE estimator.  

Throughout this section, we take
$\cnorm{\cdot}=\|\cdot\|_\infty$.
We write $\tau_w(\I)$ for $\tau_\diamond(\I)$ to stress
its dependence on the weights.
Motivated by the role of this threshold in
Theorem~\ref{cor:main}, we consider maximizing $\tau_w$
as a geometric criterion for choosing the weights.
In practice, one rarely knows $\I^\ast$ exactly, but a domain
expert may have some prior information, e.g., that
there is a single non-null cluster. For a prescribed collection
$\mathcal E$ of target patterns, we therefore consider
\[
w^\ast\in
\argmax_{w\in\mathcal W}\min_{\I\in\mathcal E}\tau_w(\I),
\]
where $\mathcal W$ is a class of normalized weights
$1=w_1\geq\cdots\geq w_m\geq0$, with the additional
restrictions specified in  results below.
 Note that the weights optimized for estimation or FDR control (e.g.  \cite{bogdan2015slope}) will generally differ from those that maximize probability of pattern recovery. 

Proofs of the following results are given in the Appendix \ref{app:SLOPEproofs}.
\begin{lemma}[No prior knowledge]\label{lem:D1}
Assume $1=w_1>w_2>\cdots>w_m>0$.
If $\mathcal E=\mathcal I$, then
	\[
	\min_{\I\in\mathcal{I}} \tau_w(\I) = \min \left\{\min_{i=1,\ldots,m-1}  \frac{w_i-w_{i+1}}{2} ,w_m\right\}.
	\]
The weights maximizing this minimum over the stated class
are given by 
	$w_i^\ast = \frac{m+\frac12-i}{m-\frac12}$, $i=1,\ldots,m$.
Moreover,
$\min_{\I\in\mathcal{I}} \tau_{w^\ast}(\I)  = \frac{1}{2m-1}$.
\end{lemma}

The weights $w^*$ from the last Lemma coincide with the OSCAR penalty \cite{OSCAR08}: we have
\[
\|x\|_{w^\ast} = \sum_{i=1}^m(c(i-1)+1)|x|_{(i)}= \sum_{j=1}^m |x_j| + c \sum_{1\leq j<k\leq m} \min\left\{ |x_j|, |x_k|\right\}
\]
with $c=-1/(m-1/2)$.

For strictly decreasing positive weights, the SLOPE pattern $\I_x\in \mathcal{I}$ is conveniently identified with vector $M_x\in\R^m$ defined by 
$(M_x)_i = \mathrm{sign}(x_i )\mathrm{rank}(|x|)_i$ for $i=1,\ldots,m$,
see e.g. \cite[Definition 2.1]{7aut}. We write $\tau_w(M_x)=\tau_w(\I_x)$. By the symmetry of the SLOPE norm, one has $\tau_w(M) = \tau_w( \pm M_{\sigma(1)},\ldots, \pm M_{\sigma(m)})$. 
Let $w_{(i,j)}$ denote the average $\frac{1}{j-i+1} \sum_{k=i}^j w_k$ and let $a^{(k)}$ denote the vector $(a,\ldots,a)\in\R^k$.  Below we present the thresholds $\tau_w$ calculated for all patterns with unique non-zero cluster.
\begin{lemma}[A single nonzero cluster]
\label{lem:SLOPEtuning}
Let $m\geq2$ and fix
$k\in\{1,\ldots,m\}$.
Let $\mathcal W_k$ be the class of weights satisfying
\[
    1=w_1\geq\cdots\geq w_m\geq0,\qquad     w_1>w_k \quad\text{if }k\geq2,
    \qquad
    w_{k+1}>0 \quad\text{if }k<m.
\]
Put $M_k=(1^{(k)},0^{(m-k)})$. 
For every $w\in\mathcal W_k$,
\[
    \tau_w(M_k)
    =
    \begin{cases}
        w_{(2,m)}, & k=1,\\[1mm]
        \min\left\{
            h(w_1,\ldots,w_k),\,w_{(k+1,m)}
        \right\}, & 2\leq k<m,\\[1mm]
        h(w_1,\ldots,w_m), & k=m,
    \end{cases}
\]
where for $k\geq2$,
\[
    h(w_1,\ldots,w_k)
    =
    \begin{cases}
        w_{(1,i)}-w_{(1,k)}, & k=2i,\\[1mm]
        \min\left\{
            w_{(1,i)}-w_{(1,k)},
            w_{(1,k)}-w_{(i+2,k)}
        \right\}, & k=2i+1.
    \end{cases}
\]
Moreover,
\[
    \max_{w\in\mathcal W_k}\tau_w(M_k)
    =
    \begin{cases}
        1, & k=1,\\
        1/3, & 2\leq k<m,\\
        1/2, & k=m.
    \end{cases}
\]
A maximizing choice of weights is
\[
    w^\ast
    =
    \begin{cases}
        1^{(m)},
            & k=1,\\[1mm]
        \bigl(1^{(i)},(1/3)^{(m-i)}\bigr),
            & k=2i<m,\\[1mm]
        \bigl(1^{(i)},2/3,(1/3)^{(m-i-1)}\bigr),
            & k=2i+1<m,\quad i\geq1,\\[1mm]
        \bigl(1^{(i)},0^{(i)}\bigr),
            & k=m=2i,\\[1mm]
        \bigl(1^{(i)},1/2,0^{(i)}\bigr),
            & k=m=2i+1.
    \end{cases}
\]
\end{lemma}

 Assume that $\mathcal{M} = \{M\colon \|M\|_\infty =1\}$, which corresponds to all patterns with a single non-zero cluster. Then, with $m=2^k+i$, i.e. $k=\lfloor \log_2(m)\rfloor$, we can show that
	\[
	\max_{w\colon w_1=1} \min_{M\in\mathcal{M}} \tau_w(M) = \frac{2^k}{2^k(k+2)+i} \approx \frac{1}{\log_2(m)+2}.
	\]

\section*{Acknowledgements}
The authors are grateful to Ma{\l}gorzata Bogdan for her inspiration in initiating this research and to Piotr Zwiernik for his guidance in identifying pertinent references on atomic norms.

We thank the two referees for their helpful comments, which improved the presentation and clarity of the paper.

\bibliographystyle{plainnat}
\bibliography{biblio} 

\appendix
\section{Proofs}\label{app:proofs}
\subsection{Proofs from Section~\ref{sec:preli}}
\begin{proof}[Proof of Lemma~\ref{lem:restr_ball}]
	For $y\in\SS_{\I}$, using $\PP_{\I}^\top=\PP_{\I}$, we have 
	\begin{align*}
		\polydr{y}{\I} &= \sup\{x^\top y\colon x\in \SS_{\I}, \max_{i=1,\ldots,K} v_i^\top x\leq 1\} \\
        &= \sup\{(\PP_{\I} x)^\top y\colon x\in \R^m, \max_{i=1,\ldots,K} v_i^\top \PP_{\I} x\leq 1\}  \\
		&=\sup\{x^\top y\colon x\in \R^m, \max_{i=1,\ldots,K} (\PP_{\I} v_i)^\top x\leq 1\}.
	\end{align*}
	By the standard polytope-duality argument, we obtain the assertion.
\end{proof}

\begin{proof}[Proof of Lemma~\ref{lem:tau}]
Let
\[
d
= \inf_{\pi^\bot\in\SS_\I^\bot\colon \polyd{f_\I+\pi^\bot}> 1}\cnorm{D\pi^\bot} ,
\]
and
\[
\mathcal A
=
\left\{
t\geq0\colon 
\forall\,\pi^\bot\in\SS_\I^\bot,\quad
\cnorm{D\pi^\bot}\leq t
\ \Longrightarrow\
\polyd{f_\I+\pi^\bot}\leq1
\right\}.
\]
By Definition~\ref{def:tau},
\[
\tau_\diamond(\I)=\sup\mathcal A.
\]
We show that $\tau_\diamond = d$. 

Since \(D\) is injective and \(\polyd{f_\I}=1\), we have
\(0\in\mathcal A\). Let \(0\leq t<d\). If
\(\pi^\bot\in\SS_\I^\bot\) satisfies
\(\cnorm{D\pi^\bot}\leq t\), then
\(\polyd{f_\I+\pi^\bot}\leq1\), by the definition of \(d\).
Thus \(t\in\mathcal A\), and hence
\[
\tau_\diamond(\I)=\sup\mathcal A\geq d.
\]

Conversely, suppose that \(d<+\infty\) and let \(t>d\). By the
definition of the infimum, there exists
\(\pi^\bot\in\SS_\I^\bot\) such that
\[
\cnorm{D\pi^\bot}<t
\qquad\text{and}\qquad
\polyd{f_\I+\pi^\bot}>1.
\]
Therefore \(t\notin\mathcal A\), which implies
\(\sup\mathcal A\leq d\). If \(d=+\infty\), the set of such bad
perturbations is empty, so every \(t\geq0\) belongs to
\(\mathcal A\) and \(\tau_\diamond(\I)=+\infty=d\).
Consequently, $\tau_\diamond(\I) = d$. 
\end{proof}

\begin{proof}[Proof of Lemma~\ref{lem:tau_positive}]
If \( |\I|=1 \), then \(F_\I=\{f_\I\}\), \(\SS_\I^\perp=\{0\}\), and
\(\tau_\diamond(\I)=+\infty\); hence the assertion holds.

Suppose that \( |\I|\geq2 \). The direction space of \(F_\I\) is
\[
    \operatorname{lin}(F_\I-F_\I)
    =
    \operatorname{span}\{v_i-v_j:i,j\in\I\}
    =
    \SS_\I^\perp.
\]
Since \(f_\I=\PP_\I v_{i_0}\), we therefore have
\[
    \operatorname{aff}(F_\I)=f_\I+\SS_\I^\perp.
\]
Moreover, because \(F_\I\) is a face of the polytope \(B^\ast\),
\[
    B^\ast\cap\operatorname{aff}(F_\I)=F_\I.
\]

The map \(\pi^\bot\mapsto\cnorm{D\pi^\bot}\) is a norm on
\(\SS_\I^\perp\). Consequently, for \(t>0\),
\[
\begin{split}
&\polyd{f_\I+\pi^\bot}\leq1
\quad\text{for every }\pi^\bot\in\SS_\I^\perp
\text{ with }\cnorm{D\pi^\bot}\leq t\\
&\qquad\Longleftrightarrow\qquad
f_\I+
\bigl\{\pi^\bot\in\SS_\I^\perp:
\cnorm{D\pi^\bot}\leq t\bigr\}
\subseteq F_\I.
\end{split}
\]
By the relative-interior neighborhood characterization, such a \(t>0\)
exists if and only if \(f_\I\in\mathrm{ri}(F_\I)\). By Definition~\ref{def:tau},
this is equivalent to \(\tau_\diamond(\I)>0\).
\end{proof}

\begin{proof}[Proof of Lemma~\ref{lem:zeta}]
Let
\[
d
=
\inf_{\substack{K\in\Mast\\
			\patt{K}\neq\patt{K^\ast}}} \cnorm{\Gamma^\ast\v{K-K^\ast}}.
\]
We first show that $\zeta_\diamond(K^\ast)=d$. 
If \(z<d\), then every \(K\in\Mast\) satisfying
\[
\cnorm{\Gamma^\ast\v{K-K^\ast}}\leq z
\]
has the same pattern as \(K^\ast\). Hence \(z\) is admissible in
the definition of \(\zeta_\diamond\), and
\(\zeta_\diamond(K^\ast)\geq d\).

Conversely, if \(z>d\), there exists \(K\in\Mast\) such that
\[
\patt{K}\neq\patt{K^\ast},
\qquad
\cnorm{\Gamma^\ast\v{K-K^\ast}}<z.
\]
Thus, \(z\) is not admissible, and \(\zeta_\diamond(K^\ast)\leq d\),
which ends the proof of the first part.

If $\cnorm{\cdot}=\|\cdot\|_\infty$, then 
    \begin{align*}
\zeta_\diamond(K^\ast) 
& = 
\inf_{\substack{K\in\Mast\\
      \patt{K}\neq\patt{K^\ast}}} \|\Gamma^\ast\v{K-K^\ast}\|_\infty\\
      &\geq  \opnorm{K^\ast}_\infty^{-2}\inf_{\substack{K\in\Mast\\
      \patt{K}\neq\patt{K^\ast}}} \!\!\! \|\v{K-K^\ast}\|_\infty\\
      & \geq \opnorm{K^\ast}_\infty^{-2}\!\!\!\inf_{\substack{x\in\SS^\ast\\
      \patt{x}\neq\patt{x^\ast}}} \!\!\! \|x-x^\ast\|_\infty,
\end{align*}
where $x^\ast=\vp{K^\ast}$ and $\SS^\ast=\SS_{x^\ast}$. 
\end{proof}

\subsection{Proofs from Section~\ref{sec:ProofSkeleton}}\label{sec:remaining}

\begin{proof}[Proof of Lemma~\ref{lem:1}]
	(i) The proof follows essentially the same steps as in \cite[Lemma 3]{ravikumar2011graphical}. By the Hadamard inequality, one shows that the objective function is coercive. Since the function is convex, the minimum is attained and is also unique.
	
	(ii) The existence and uniqueness of the solution follow from the same argument as in (i).
	
	Define the function $L(K) = \tr{\hat{\Sigma}K}-\log\det(K)$. 
	A standard result shows that the directional derivative $DL$ of $L$ at $K\in\Mast\cap\mathrm{Sym}_+(p)$ in the direction $H\in\Mast$ is given by 
	\begin{align*}
		\scalar{DL(K)}{H} &= \lim_{t\to 0}\frac{L(K+tH)-L(K)}{t}=\tr{\hat{\Sigma}\cdot H}-\tr{K^{-1}\cdot H} \\
		&= \scalar{\proj(\hat{\Sigma}-K^{-1})}{H}.
	\end{align*}
	Since $L+\lambda \norms$ is convex on $\Mast$, the minimizer $\tilde{K}$ is characterized by the condition $0\in \partial (L+\lambda \norms)(\tilde{K})$. Since $L$ is differentiable, this condition simplifies to
	\[
	\proj(\tilde{K}^{-1}-\hat{\Sigma})= -DL(\tilde K)\in \lambda\, \partial \norms(\tilde{K}).
	\]
\end{proof}

\begin{proof}[Proof of Lemma~\ref{lem:hatKtildeK}]
	Assume $\normd(\Pi)\leq 1$. 
	By Lemma~\ref{lem:1} (ii), we have $\proj(\Pi)=\tilde{\Pi}\in \partial \norms(\tilde{K})$. Applying Lemma \ref{lem:subd}, we obtain
	\[
	\normsd(\tilde{\Pi})\leq 1\quad\mbox{ and }\quad  \norms(\tilde{K}) = \scalar{\tilde{\Pi}}{\tilde{K}}.
	\]
	By the definitions of $\norms = \norm\big|_{\Mast}$ and $\proj$, it follows that
	\[
	\norm(\tilde{K}) = \norms(\tilde{K})   =\tr{\tilde{\Pi}\cdot \tilde{K}}=\tr{\Pi\cdot \tilde{K}}.
	\]
	Now, if $\normd(\Pi)\leq 1$, then by Lemma \ref{lem:subd}, we conclude that
	$\Pi\in \partial \norm(\tilde{K})$. By the uniqueness of the solution to \eqref{eq:hatK2} (recall Lemma~\ref{lem:1}), this implies that $\hat{K}=\tilde{K}$.
	
	Assume $\hat{K}=\tilde{K}$. Then $\Pi=\hat{\Pi}$, which implies that $\normd(\Pi) = \normd(\hat{\Pi})\leq 1$.
\end{proof}

\begin{proof}[Proof of Lemma~\ref{lem:Vbound}]
	If $|\I^\ast|=1$, then $\Mast=\mathrm{Sym}(p)$ and
	$\proj$ is the identity. Hence $\Pi=\tilde\Pi\in
	\partial\norm(\tilde K)$, so $\normd(\Pi)\leq1$.
	We may therefore assume that $|\I^\ast|\geq2$.

	From the definition of $\Pi$, setting $V=R+\Sigma^\ast-\hat{\Sigma}$, we obtain
	\[
	\Pi = \frac{1}{\lambda}\left(\tilde{K}^{-1}-\hat{\Sigma}\right)= \frac{1}{\lambda}\left(V-\Sigma^\ast (\tilde{K}-K^\ast)\Sigma^\ast\right). 
	\]
	Moreover, we can decompose
	$\v{\Pi} = \Q \,\v{\Pi}+(I_{p^2}-\Q)\v{\Pi}$. 
	A key property of $\Q$ is
	$(I_{p^2} -\Q)\Gamma^\ast\PM = 0$.
	Thus, we have
	\[
	(I_{p^2}-\Q)\v{\Pi} = \frac{1}{\lambda}(I_{p^2}-\Q) \v{V},
	\]
	since $\tilde K-K^\ast\in\Mast$, and hence
	\[
	(I_{p^2}-\Q)\v{\Sigma^\ast(\tilde K-K^\ast)\Sigma^\ast}
	=
	(I_{p^2}-\Q)\Gamma^\ast\PM\v{\tilde K-K^\ast}
	=
	0.
	\]
	
	Since $\proj(\Pi)=\tilde{\Pi}$, we have $\PM\v{\Pi}=\v{\proj(\Pi)}=\v{\tilde\Pi}$.
	By assumption we have $\patt{\tilde{K}}=\patt{K^\ast}$, which, by definition, implies $\partial \norms(\tilde{K}) = \partial \norms(K^\ast)$. By Definition~\ref{def:f}, we have 
	\[
	\vp{\proj( \partial \norms(K^\ast))}= \PP_{\I^\ast}\vp{ \partial \norms(K^\ast)}  = \{ f_{\I^\ast}\},
	\]
	which implies that $\v{\tilde{\Pi}} = D f_{\I^\ast}$, where $D$ is the duplication matrix. 
	Therefore,
	\begin{align*} 
		(I_{p^2}-\PM) \v{\Pi} & =(\Q-\PM)\v{\Pi} + (I_{p^2}-\Q) \v{\Pi} \\
		&=(\Q-\PM)D f_{\I^\ast} + \frac{1}{\lambda}(I_{p^2}-\Q) \v{V},
	\end{align*}
	where we used the fact that $\Q-\PM =(\Q-\PM)\PM$.
	Since $\vp{\Pi} = f_{\I^\ast} + \pi^\bot$ with $\pi^\bot=(I_m-\PP_{\I^\ast})\vp{\Pi}\in\SS_{\I^\ast}^\bot$, by Definition~\ref{def:tau}, it suffices to show that 
	\[
	\cnorm{\v{\proj^\bot(\Pi)}} = \cnorm{D\pi^\bot}\leq \tau_\diamond(\I^\ast),
	\]
	to obtain $\normd(\Pi)=\polyd{\vp{\Pi}}=\polyd{f_{\I^\ast}+\pi^\bot}\leq 1$. From the previous calculations, we obtain
	\begin{align*}
		\cnorm{{(I_{p^2}-\PM)\v{\Pi}}} &\leq  \cnorm{(\Q-\PM) D f_{\I^\ast}}  +  \frac{1}{\lambda}\cnorm{ (I_{p^2}-\Q) \v{V}}  \\
		&\leq \cnorm{(\Q-\PM) D f_{\I^\ast}} + \frac{1}{\lambda}\copnorm{I_{p^2}-\Q}\cnorm{ \v{V}}.
	\end{align*}
    	By the irrepresentability condition,
	\[
	\cnorm{(\Q-\PM)D f_{\I^\ast}}
	\leq
	(1-\alpha)\tau_\diamond(\I^\ast).
	\]
	Moreover, since
	$V=R+\Sigma^\ast-\hat\Sigma$, the assumption of the lemma gives
	\[
	\frac1\lambda
	\copnorm{I_{p^2}-\Q}\cnorm{\v V}
	\leq
	\alpha\tau_\diamond(\I^\ast).
	\]
	Consequently,
	\[
	\cnorm{(I_{p^2}-\PM)\v\Pi}
	\leq
	\tau_\diamond(\I^\ast),
	\]
	and hence $\normd(\Pi)\leq1$.
\end{proof}

\begin{proof}[Proof of Lemma~\ref{lem:Rbound}]
	By Definition~\ref{def:c}, we have
	\[
	\copnorm{\Sigma^\ast\Delta\otimes I_p} \leq \cH \cnorm{\Gamma^\ast\v{\Delta}}\leq r<1. 
	\]
	Thus, the matrix $(I_p+\Delta\Sigma^\ast)$ is invertible.
	By the definition of $R(\Delta)$, we have $ (K^\ast+\Delta)^{-1}-\Sigma^\ast = R(\Delta) -\Sigma^\ast\Delta\Sigma^\ast$.  On the other hand, 
    \[
(K^\ast+\Delta)^{-1}-\Sigma^\ast
=
-(K^\ast+\Delta)^{-1}\Delta\Sigma^\ast
=
-(R(\Delta)+\Sigma^\ast-\Sigma^\ast\Delta\Sigma^\ast)
\Delta\Sigma^\ast,
\]
	which gives 
	\begin{align*}
		R(\Delta)
		=\Sigma^\ast\Delta\Sigma^\ast\Delta\Sigma^\ast(I_p+\Delta\Sigma^\ast)^{-1} 
	\end{align*}
	By vectorization, we get
	\begin{align*}
		\v{R(\Delta)} =  (I_p\otimes \Sigma^\ast\Delta ) ((I_p+\Sigma^\ast\Delta)^{-1}\otimes I_p)\Gamma^\ast \v{\Delta}.
	\end{align*}
	Since $\copnorm{\Sigma^\ast\Delta\otimes I_p} \leq  r<1$, we have
	\[
	\copnorm{(I_p+\Sigma^\ast\Delta )^{-1}\otimes I_p} \leq \sum_{k=0}^\infty \copnorm{\Sigma^\ast\Delta\otimes I_p}^k \leq \frac{1}{1-r}.
	\]
	Therefore
	\begin{align*}
		\copnorm{(I_p\otimes \Sigma^\ast\Delta ) ((I_p+\Sigma^\ast\Delta)^{-1}\otimes I_p) } \leq \frac{\cH \cnorm{\Gamma^\ast\v{\Delta}}}{1-\cH \cnorm{\Gamma^\ast\v{\Delta}}}
	\end{align*}
    and Lemma~\ref{lem:Rbound} follows.
\end{proof}

Define the function $G\colon \Mast\cap\mathrm{Sym}_+(p)\to\v{\Mast}$ by 
\begin{align*}
	G(K) &= \PM\v{K^{-1}-\tilde{K}^{-1}},
\end{align*}
so that 
\[
F(\Gamma^\ast \v{\Delta}) = \Q G(K^\ast+\Delta)+\Gamma^\ast\v{\Delta}.
\]
First we prove the following easy lemma. 
\begin{lemma}
	$G(K)=0$ if and only if $K=\tilde{K}$.
\end{lemma}
\begin{proof}
	We rewrite $G(K)$ as 
	\[
	G(K)= \PM\v{K^{-1}-\hat{\Sigma}-\lambda \Pi} = \v{ \proj(K^{-1}-\hat{\Sigma}) - \lambda \tilde{\Pi}}.
	\]
	By the uniqueness of the solution to \eqref{eq:tildeK}, it follows that $G(K)=0$ if and only if $K=\tilde{K}$.
\end{proof}

Now we are ready to present the proof of Lemma~\ref{lem:Ffixed}. 
\begin{proof}[Proof of Lemma~\ref{lem:Ffixed}]
	By definition, $F(\Gamma^\ast\v{\Delta})=\Gamma^\ast\v{\Delta}$ if and only if
	\[
    (\Gamma^\ast_{1,1})^+ G(K^\ast+\Delta)=0.
    \]
    This is equivalent to the existence of $w\in\R^{p^2}$ such that 
	$$G(K^\ast+\Delta)=(I_{p^2} - \Gamma^\ast_{1,1}(\Gamma^\ast_{1,1})^+)w.$$
	We now show that $(I_{p^2} - \Gamma^\ast_{1,1}(\Gamma^\ast_{1,1})^+)w = 0$.
	Since $\Gamma_{1,1}^\ast(\Gamma_{1,1}^\ast)^{+}=\PM$, it follows that
	$$
	(I_{p^2} - \PM)w = G(K^\ast+\Delta) = \PM\v{(K^\ast+\Delta)^{-1}-\tilde{K}^{-1}}.
	$$
	Multiplying both sides from the left by $(I_{p^2} -\PM)$, 
	we conclude that $(I_{p^2} - \Gamma^\ast_{1,1}(\Gamma^\ast_{1,1})^+)w$ is a zero vector.
	This implies that $G(K^\ast+\Delta)=0$, i.e., $\Delta = \tilde{K}-K^\ast$. 
\end{proof}

\begin{proof}[Proof of Lemma~\ref{lem:Deltabound}]
	Let $\Delta\in\Mast$ and recall that $R(\Delta) = (K^\ast+\Delta)^{-1}-\Sigma^\ast+\Sigma^\ast\Delta\Sigma^\ast$.
	Then, we compute
	\begin{align*}
		G(K^\ast+\Delta) &= \PM\v{(K^\ast+\Delta)^{-1}-(K^\ast)^{-1} + \Sigma^\ast-\hat{\Sigma}-\lambda\,\tilde\Pi}\\ 
		&=\PM\v{R(\Delta) - \Sigma^\ast\Delta\Sigma^\ast+ (\Sigma^\ast-\hat{\Sigma})-\lambda\,\tilde\Pi}\\
		&=\PM\v{R(\Delta) + (\Sigma^\ast-\hat{\Sigma})-\lambda\,\tilde\Pi} - \PM\Gamma^\ast\v{\Delta}.
	\end{align*}
	Since $\v{\Delta}=\PM\v{\Delta}$, we have $\PM\Gamma^\ast\v{\Delta}=\Gamma_{1,1}^\ast \v{\Delta}$. 
	Thus,
	\begin{align}\label{eq:F}\begin{split}
			F(\Gamma^\ast\v{\Delta}) =& \Q \left( \v{R(\Delta)} +\v{\Sigma^\ast-\hat\Sigma}-\lambda\,\v{\tilde\Pi}\right)\\
			& -\Gamma^\ast(\Gamma^\ast_{1,1})^{+}\Gamma_{1,1}^\ast  \v{\Delta} + \Gamma^\ast\v{\Delta} \\
			&= \Q\left( \v{R(\Delta)} +\v{\Sigma^\ast-\hat\Sigma}-\lambda\,\v{\tilde\Pi}\right).
	\end{split}\end{align}
	By Definition~\ref{def:c}, we have $\cnorm{\v{\tilde{\Pi}}}\leq c_{\diamond} \polydr{\vp{\tilde{\Pi}}}{\I^\ast} = c_{\diamond}$. 
	Thus, taking norms of \eqref{eq:F}, we obtain
	\[
	\cnorm{F(\Gamma^\ast\v{\Delta})} \leq \copnorm{\Q}\left(
	\cnorm{\v{R(\Delta)}} +\cnorm{ \v{\hat\Sigma-\Sigma^\ast}}+\lambda\, c_\diamond\right).
	\]
	
	Now, suppose that $\cnorm{\Gamma^\ast\v{\Delta}}\leq r/\cH$. By Definition~\ref{def:c}, we have $\copnorm{\Sigma^\ast\Delta\otimes I_p}\leq \cH\cnorm{\Gamma^\ast\v{\Delta}}\leq r$; thus  $F(\Gamma^\ast\v{\Delta})$ is well defined under the condition $r<1$.
	Applying Lemma~\ref{lem:Rbound} and using assumption \eqref{eq:smallr}, we obtain 
	\begin{align*}
		\cnorm{F(\Gamma^\ast\v{\Delta})}\leq  \copnorm{\Q}\left( \frac{1}{\cH}\frac{r^2}{1-r} + \cnorm{\v{\hat{\Sigma}-\Sigma^\ast}}+\lambda\, c_\diamond\right)< r/\cH. 
	\end{align*}
	Since $F$ is continuous, Brouwer's fixed-point theorem guarantees the existence of a fixed point $\Gamma^\ast\v{\tilde\Delta}$ satisfying
	\[
	\cnorm{\Gamma^\ast\v{\tilde\Delta}}< r/\cH. 
	\]
    Thus the closed ball is mapped into its interior.
	Moreover, by Lemma~\ref{lem:Ffixed}, this fixed point is unique. By Lemma~\ref{lem:Ffixed}, we also have $\tilde{\Delta}=\tilde{K}-K^\ast\in\Mast$. The second part of the assertion follows from the definition of $\zeta_\diamond(K^\ast)$. 
\end{proof}

\subsection{Proofs from Section~\ref{sec:SLOPE}}
\label{app:SLOPEproofs}

Throughout, $\poly{\cdot}=\|\cdot\|_w$ is the SLOPE norm on $\R^m$, where $m\geq2$ and $1=w_1\geq w_2\geq\cdots\geq w_m\geq0$, and $\cnorm{\cdot}=\|\cdot\|_\infty$.
The proofs rest on three observations: $\tau_w(\I)$ is the radius of the largest $\ell_\infty$-ball centred at $f_\I$ that fits into the face $F_\I$ (Lemma~\ref{lem:app-tau-inradius}); the inradii of permutahedra are explicit (Lemma~\ref{lem:app-perm}); and every face of the dual SLOPE ball is a product of permutahedra (Lemma~\ref{lem:app-faces}).

We write $\mathbf 1$ for a vector of ones and $\bar a=d^{-1}\sum_{i=1}^da_i$ for the mean of $a\in\R^d$; as in Section~\ref{sec:SLOPE}, $w_{(i,j)}$ denotes the mean of $w_i,\ldots,w_j$.
The permutahedron and the signed permutahedron of $a\in\R^d$ are the convex hulls of all permutations, respectively of all signed permutations, of $a$:
\begin{align*}
 \mathcal P(a)&=\mathrm{conv}\bigl\{(a_{\sigma(1)},\ldots,a_{\sigma(d)})\colon\sigma\in\mathfrak S_d\bigr\},\\
 \mathcal P_\pm(a)&=\mathrm{conv}\bigl\{(\varepsilon_1a_{\sigma(1)},\ldots,\varepsilon_da_{\sigma(d)})\colon\sigma\in\mathfrak S_d,\ \varepsilon\in\{-1,1\}^d\bigr\},
\end{align*}
where $\mathfrak S_d$ is the group of permutations of $\{1,\ldots,d\}$.
By the rearrangement inequality, $\|x\|_w=\max_v v^\top x$, the maximum being taken over all signed permutations $v$ of $w$; hence the dual unit ball is the signed permutahedron $B^\ast=\mathcal P_\pm(w)$ \citep[Section~2.2]{7aut}.

For a polytope $P$ and a point $c\in P$, the relative $\ell_\infty$-inradius of $P$ at $c$ is
\[
 \mathrm{rad}(P,c)=\sup\bigl\{t\geq0\colon\ y\in\mathrm{aff}(P)\text{ and }\|y-c\|_\infty\leq t\ \implies\ y\in P\bigr\},
\]
that is, the radius of the largest $\ell_\infty$-ball centred at $c$ that fits into $P$ when both are viewed inside the affine hull $\mathrm{aff}(P)$. We put
\[
 \rho(a)=\mathrm{rad}\bigl(\mathcal P(a),\bar a\mathbf 1\bigr),
 \qquad
 \rho_\pm(a)=\mathrm{rad}\bigl(\mathcal P_\pm(a),0\bigr).
\]
Both radii are invariant under permutations of $a$, and they equal $+\infty$ when the polytope is a single point.

\begin{lemma}\label{lem:app-tau-inradius}
Let $\poly{\cdot}$ be any atomic norm and $\cnorm{\cdot}=\|\cdot\|_\infty$. If $\emptyset\neq\I\in\mathcal I$ and $f_\I\in F_\I$, then
\[
 \tau_\diamond(\I)=\mathrm{rad}(F_\I,f_\I).
\]
\end{lemma}
\begin{proof}
Since $\SS_\I^\perp$ is the direction space of $F_\I$
and $f_\I\in F_\I$, we have
\[
\mathrm{aff}(F_\I)=f_\I+\SS_\I^\perp.
\]
Moreover,
$B^\ast\cap\mathrm{aff}(F_\I)=F_\I$, because $F_\I$
is a face of $B^\ast$.
Finally, $D$ duplicates the off-diagonal coordinates, so that
$\|Du\|_\infty=\|u\|_\infty$.
Thus the admissible radii in the definitions of
$\tau_\diamond(\I)$ and $\mathrm{rad}(F_\I,f_\I)$ coincide.
\end{proof}

\begin{lemma}\label{lem:app-perm}
\begin{enumerate}
\item[(i)] Let $a_1\geq\cdots\geq a_d$ with $a_1>a_d$, and for $0<j<d$ let
\[
 \Delta_j(a)=\sum_{i=1}^j(a_i-\bar a)=\frac1d\sum_{i=1}^{j}\sum_{k=j+1}^{d}(a_i-a_k)
\]
be the total excess of the $j$ largest entries over the mean. Then
\begin{equation}\label{eq:app-perm-radius}
 \rho(a)=\min_{0<j<d}\frac{\Delta_j(a)}{\min\{j,d-j\}}=h(a),
\end{equation}
with $h$ as in Lemma~\ref{lem:SLOPEtuning}. In particular, $\rho(a)\leq\Delta_1(a)=a_1-\bar a$ and $\rho(a)\leq\Delta_{d-1}(a)=\bar a-a_d$, so that $\rho(a)\leq(a_1-a_d)/2$.
\item[(ii)] Let $b_1\geq\cdots\geq b_d\geq0$ with $b_1>0$. Then $\rho_\pm(b)=\bar b$.
\end{enumerate}
\end{lemma}

\begin{proof}
(i) By Rado's theorem \citep{Rado}, $y\in\mathcal P(a)$ if and only if $\mathbf 1^\top y=\mathbf 1^\top a$ and $\sum_{i\in J}y_i\leq\sum_{i=1}^{|J|}a_i$ for every $J\subset\{1,\ldots,d\}$. With $y=\bar a\mathbf 1+u$, these conditions read $\mathbf 1^\top u=0$ and $\sum_{i\in J}u_i\leq\Delta_{|J|}(a)$ for $0<|J|<d$. For $|J|=j$ and $t\geq0$,
\[
 \max\Bigl\{\sum_{i\in J}u_i\colon \mathbf 1^\top u=0,\ \|u\|_\infty\leq t\Bigr\}=t\min\{j,d-j\}:
\]
the sum is at most $tj$ and, being equal to $-\sum_{i\notin J}u_i$, at most $t(d-j)$; for $j\leq d/2$ the value $tj$ is attained at $u=t$ on $J$ and $u=-tj/(d-j)$ off $J$, and symmetrically for $j>d/2$. Hence $\bar a\mathbf 1+u\in\mathcal P(a)$ for all $u\in\mathbf 1^\perp$ with $\|u\|_\infty\leq t$ if and only if $t$ does not exceed the middle term of \eqref{eq:app-perm-radius}. This term is positive, since $\Delta_j(a)\geq(a_1-a_d)/d>0$. Consequently, $\mathrm{aff}\,\mathcal P(a)$ is the hyperplane $\{y\colon\mathbf 1^\top y=d\bar a\}$ and the middle term equals $\rho(a)$.

For the second equality, let $A_\ell$ and $L_\ell$ be the means of the $\ell$ largest and of the $\ell$ smallest entries of $a$. The quotient in \eqref{eq:app-perm-radius} equals $A_j-\bar a$ for $j\leq d/2$ and $\bar a-L_{d-j}$ for $j\geq d/2$. As $A_\ell$ is nonincreasing and $L_\ell$ is nondecreasing in $\ell$, the minimum is $\min\{A_{\lfloor d/2\rfloor}-\bar a,\ \bar a-L_{\lfloor d/2\rfloor}\}$, which is $h(a)$ (for even $d$ the two terms coincide).

(ii) By the formula for the dual SLOPE norm with weights $b$ \citep[Section~2.2]{7aut}, $z\in\mathcal P_\pm(b)$ if and only if $\sum_{i=1}^j|z|_{(i)}\leq\sum_{i=1}^jb_i$ for $j=1,\ldots,d$. The cube $\{z\colon\|z\|_\infty\leq t\}$ satisfies all these inequalities if and only if its vertex $t\mathbf 1$ does, i.e.\ if and only if $t\leq\frac1j\sum_{i=1}^jb_i$ for all $j$. Prefix means of a nonincreasing sequence are nonincreasing, so the binding constraint is $j=d$. Hence $\mathcal P_\pm(b)$ is full-dimensional and $\rho_\pm(b)=\bar b$.
\end{proof}

\begin{lemma}\label{lem:app-faces}
Let $x\in\R^m$ and $s_i=\mathrm{sign}(x_i)$. Let $G_1,\ldots,G_r$ be the clusters of $x$, i.e.\ the maximal sets of indices on which $|x|$ takes a common nonzero value, ordered so that this value decreases in $j$, and let $G_0=\{i\colon x_i=0\}$. Split $\{1,\ldots,m\}$ into consecutive blocks $B_1,\ldots,B_r,B_0$, in this order, with $|B_j|=|G_j|$. Then, with $f_x=f_{\I_x}$:
\begin{enumerate}
\item[(i)] $y\in F_x$ if and only if $(s_iy_i)_{i\in G_j}\in\mathcal P(w_{B_j})$ for $j=1,\ldots,r$ and $y_{G_0}\in\mathcal P_\pm(w_{B_0})$;
\item[(ii)] $f_x$ equals $s_i\,\bar w_{B_j}$ at $i\in G_j$ $(j\geq1)$ and $0$ on $G_0$; in particular, $f_x\in F_x$;
\item[(iii)] $\tau_w(\I_x)=\min\bigl\{\rho(w_{B_1}),\ldots,\rho(w_{B_r}),\ \rho_\pm(w_{B_0})\bigr\}$.
\end{enumerate}
\end{lemma}

For example, for $x=(0,-3,1,3)$ we have $G_1=\{2,4\}$, $G_2=\{3\}$, $G_0=\{1\}$, the face is $F_x=\{y\colon(-y_2,y_4)\in\mathcal P(w_1,w_2),\ y_3=w_3,\ |y_1|\leq w_4\}$, and $\tau_w(\I_x)=\min\{(w_1-w_2)/2,\,w_4\}$ provided $w_1>w_2$ and $w_4>0$.

\begin{proof}
(i) The face $F_x=\{y\in B^\ast\colon x^\top y=\|x\|_w\}$ is the convex hull of those signed permutations $v$ of $w$ for which $x^\top v=\|x\|_w$. Since
\[
 x^\top v\leq\sum_{i=1}^m|x_i|\,|v_i|\leq\sum_{i=1}^m|x|_{(i)}w_i=\|x\|_w
\]
by the rearrangement inequality, these are exactly the $v$ with $x_iv_i\geq0$ for all $i$ and with $|v|$ ordered like $|x|$, that is, $|x_i|>|x_{i'}|\Rightarrow|v_i|\geq|v_{i'}|$. Equivalently, for each $j\geq1$ the entries $(s_iv_i)_{i\in G_j}$ are those of $w_{B_j}$ in some order, and the entries of $v_{G_0}$ are those of $w_{B_0}$ in some order and with arbitrary signs. This set of maximizers is a Cartesian product over $G_1,\ldots,G_r,G_0$, and the convex hull of a Cartesian product is the product of the convex hulls.

(ii) Let $z$ be the vector described in (ii). By (i), $z\in F_x$, because $\bar a\mathbf 1\in\mathcal P(a)$ and $0\in\mathcal P_\pm(a)$ for every $a$. If $v$ and $v'$ are maximizers as in (i), then $\sum_{i\in G_j}s_iv_i=\sum_{i\in G_j}s_iv'_i$ for $j\geq1$ (both sums equal the sum of the entries of $w_{B_j}$), whence $z^\top(v-v')=0$. Since these differences span $\SS_x^\perp$, we get $z\in\SS_x$, and therefore $f_x=\PP_{\I_x}z=z$.

(iii) By (ii) and Lemma~\ref{lem:app-tau-inradius}, $\tau_w(\I_x)=\mathrm{rad}(F_x,f_x)$. Up to a permutation of coordinates and sign changes, which are isometries of $\|\cdot\|_\infty$, parts (i) and (ii) exhibit $(F_x,f_x)$ as the product of the pairs $(\mathcal P(w_{B_j}),\bar w_{B_j}\mathbf 1)$, $j=1,\ldots,r$, and $(\mathcal P_\pm(w_{B_0}),0)$. Affine hulls and $\ell_\infty$-balls of products are products, so the relative inradius of a product is the minimum of the relative inradii of its factors.
\end{proof}

\begin{proof}[Proof of Lemma~\ref{lem:D1}]
We assume $1=w_1>w_2>\cdots>w_m>0$. Let $g=\min_{j<m}(w_j-w_{j+1})>0$.

\emph{Lower bound.} By Lemma~\ref{lem:app-faces} (iii), every $\tau_w(\I)$ is a minimum of terms of two kinds. A term $\rho_\pm(b)$, with $b$ a tail of $w$, equals $\bar b\geq w_m$. A term $\rho(a)$, with $a$ a block of $d\geq2$ consecutive weights, satisfies $\Delta_j(a)\geq\frac{j(d-j)}{d}\,g$, because each of the $j(d-j)$ differences $a_i-a_k$ in Lemma~\ref{lem:app-perm}(i) is at least $g$. Hence, by \eqref{eq:app-perm-radius},
\[
 \rho(a)\geq\min_{0<j<d}\frac{\max\{j,d-j\}}{d}\,g\geq\frac g2
\]
(blocks with $d=1$ give $\rho=+\infty$). Thus $\tau_w(\I)\geq\min\{g/2,w_m\}$ for all $\I\in\mathcal I$.

\emph{Attainment.} Let $g=w_{j_0}-w_{j_0+1}$. If $|x|_{(1)}>\cdots>|x|_{(j_0)}=|x|_{(j_0+1)}>\cdots>|x|_{(m)}>0$, the only finite term is $\rho(w_{j_0},w_{j_0+1})=g/2$. If $|x|_{(1)}>\cdots>|x|_{(m-1)}>|x|_{(m)}=0$, the only finite term is $\rho_\pm(w_m)=w_m$.

\emph{Optimal weights.} Let $t=\min\{g/2,w_m\}$. Then
\[
 1=w_1=w_m+\sum_{j=1}^{m-1}(w_j-w_{j+1})\geq t+2(m-1)t=(2m-1)t,
\]
with equality if and only if $w_m=t$ and $w_j-w_{j+1}=2t$ for all $j$, that is, $t=\frac{1}{2m-1}$ and $w_j=w_j^\ast=\frac{2(m-j)+1}{2m-1}$.
\end{proof}

\begin{proof}[Proof of Lemma~\ref{lem:SLOPEtuning}]
Let $x=(1^{(k)},0^{(m-k)})$. Lemma~\ref{lem:app-faces}, with the single cluster $G_1=\{1,\ldots,k\}$, gives
\[
 \tau_w(M)=\min\bigl\{\rho(w_1,\ldots,w_k),\ \rho_\pm(w_{k+1},\ldots,w_m)\bigr\}.
\]
We assume that $w_1>w_k$ if $k\geq2$ and that $w_{k+1}>0$ if $k<m$; otherwise the corresponding factor of $F_x$ is a single point and its term is $+\infty$. Lemma~\ref{lem:app-perm} then yields
\begin{equation}\label{eq:app-single-cluster-radius}
 \tau_w(M)=\min\bigl\{h(w_1,\ldots,w_k),\ w_{(k+1,m)}\bigr\},
\end{equation}
with the conventions $h(w_1)=w_{(m+1,m)}=+\infty$ for $k=1$ and $k=m$.

Let $t=\tau_w(M)$. If $k=1$, then $t=w_{(2,m)}\leq w_1=1$. If $k=m$, then $t\leq(w_1-w_m)/2\leq1/2$ by Lemma~\ref{lem:app-perm} (i). If $1<k<m$, then $t\leq w_{(k+1,m)}\leq w_k$ and, by Lemma~\ref{lem:app-perm} (i), $t\leq(w_1-w_k)/2=(1-w_k)/2$; adding the first bound to twice the second gives $3t\leq1$.

The weights $w^\ast$ of the lemma satisfy our assumptions and attain these bounds. For $k=1$ this is clear. For $1<k<m$, the vector $(w^\ast_1,\ldots,w^\ast_k)$ has mean $2/3$, its first and its last $\lfloor k/2\rfloor$ entries have means $1$ and $1/3$, and $w^\ast_{(k+1,m)}=1/3$; hence $t=\min\{1-\frac23,\frac23-\frac13,\frac13\}=\frac13$. For $k=m$, the three means are $1/2$, $1$ and $0$, whence $t=1/2$.
\end{proof}

Since $\tau_w$ is invariant under signed permutations of the pattern, the minimum of $\tau_w$ over the single-cluster patterns $\mathcal M$ of Section~\ref{sec:SLOPE} equals $T(w)$ defined below.

\begin{lemma}\label{lem:app-single-cluster}
Let $M_k=(1^{(k)},0^{(m-k)})$ and $T(w)=\min_{1\leq k\leq m}\tau_w(M_k)$, and let
\[
 \mathcal W_{\rm sc}=\{w\colon 1=w_1>w_2\geq\cdots\geq w_m>0\}
\]
be the set of weights satisfying the assumptions of Lemma~\ref{lem:SLOPEtuning} for every $k$. If $m=2^q+i$ with $0\leq i<2^q$, then
\begin{equation}\label{eq:app-single-cluster-optimum}
 \max_{w\in\mathcal W_{\rm sc}}T(w)=\frac{1}{q+1+m/2^q}=\frac{2^q}{2^q(q+2)+i}.
\end{equation}
\end{lemma}

\begin{proof}
Let $w\in\mathcal W_{\rm sc}$, $t=T(w)$, $N=2^q$ and $A_j=w_{(1,j)}$. Since the tail means $w_{(k+1,m)}$ are smallest for $k=m-1$, \eqref{eq:app-single-cluster-radius} gives
\[
 T(w)=\min\bigl\{w_m,\ \rho(w_1,w_2),\ \rho(w_1,w_2,w_3),\ \ldots,\ \rho(w_1,\ldots,w_m)\bigr\}.
\]

\emph{Upper bound.} For a prefix $(w_1,\ldots,w_k)$ we have $\Delta_j=j(A_j-A_k)$. Evaluating \eqref{eq:app-perm-radius} at a single index therefore gives
\begin{align*}
 t&\leq\rho(w_1,\ldots,w_{2j})\leq A_j-A_{2j} &&\text{for }1\leq j\leq N/2\quad(\text{index }j),\\
 t&\leq\rho(w_1,\ldots,w_m)\leq\tfrac{N}{m-N}\,(A_N-A_m) &&\text{if }N<m\quad(\text{index }N>m/2),\\
 t&\leq\rho(w_1,\ldots,w_m)\leq A_m-w_m &&(\text{index }m-1).
\end{align*}
Combining these bounds with $t\leq w_m$, we obtain
\begin{align*}
 1=A_1&=\sum_{s=1}^q\bigl(A_{2^{s-1}}-A_{2^s}\bigr)+(A_N-A_m)+(A_m-w_m)+w_m\\
 &\geq\Bigl(q+\frac{m-N}{N}+1+1\Bigr)t=\Bigl(q+1+\frac mN\Bigr)t,
\end{align*}
where the term $A_N-A_m$ vanishes if $N=m$.

\emph{Attainment.} Define the concave function
\[
 \Phi(x)=\sum_{\ell=0}^\infty\min\{x,\,m2^{-\ell}\},\qquad 0\leq x\leq m.
\]
Then $\Phi(0)=0$ and, for $0\leq x\leq m$,
\begin{equation}\label{eq:app-Phi}
 \Phi(1)=q+1+\frac mN,\qquad
 \Phi\Bigl(\frac x2\Bigr)=\frac{\Phi(x)+x}{2},\qquad
 \Phi(x)=m+x\ \text{ if }x\geq\frac m2:
\end{equation}
the first identity holds because $m2^{-\ell}\geq1$ exactly for $\ell\leq q$, the second follows by shifting the summation index, and the third because for $x\geq m/2$ only the term $\ell=0$ depends on $x$. 
Let $w^\ast_j=\{\Phi(j)-\Phi(j-1)\}/\Phi(1)$ for $j=1,\ldots,m$. Then $w^\ast_1=1$; $w^\ast$ is nonincreasing by concavity; $w^\ast_m=1/\Phi(1)>0$ by the third identity; and $w^\ast_1-w^\ast_2=2/\Phi(1)>0$ by the second identity with $x=2$. Hence $w^\ast\in\mathcal W_{\rm sc}$.

Fix $2\leq k\leq m$. The concave function $\varphi(x)=\Phi(x)-\frac xk\Phi(k)$ vanishes at $x=0$ and $x=k$, and $\varphi(k/2)=k/2$ by \eqref{eq:app-Phi}. Hence $\varphi(x)\geq\min\{x,k-x\}$ on $[0,k]$. Since $\Delta_j(w^\ast_1,\ldots,w^\ast_k)=\varphi(j)/\Phi(1)$, \eqref{eq:app-perm-radius} yields $\rho(w^\ast_1,\ldots,w^\ast_k)\geq1/\Phi(1)$. As $w^\ast_m=1/\Phi(1)$, we conclude that $T(w^\ast)=1/\Phi(1)$, which matches the upper bound.
\end{proof}

\section{Dual gauges}\label{app:dg}

In this section, we present general results on dual gauges and restricted dual gauges.  These results are not new, we present them for completeness. For a comprehensive treatment of these concepts, see, e.g., \cite{CONVEXANALYSIS}.

Let $V$ be a finite-dimensional vector space equipped with an inner product $\scalar{\cdot}{\cdot}$ and consider a gauge $\|\cdot\|$ on $V$; i.e., $\|\cdot\|\colon V\to[0,\infty)$ is positively homogeneous, satisfies the triangle inequality and definiteness (i.e. $\|x\|=0$ implies $x=0$). A gauge becomes a norm if additionally $\|-x\|=\|x\|$ for $x\in V$.  The dual gauge corresponding to a gauge $\|\cdot\|$ on $V$ is defined as
\[
\| y\|^\ast = \sup_{\substack{x\in V\\\|x\|\leq 1}}  \scalar{y}{x},\qquad y\in V. 
\]
This definition relies on the identification of $V$ with its dual space via the inner product.
\begin{lemma}[Characterization of the gauge subdifferential]\label{lem:subd}\ 
 Let $V$ be a vector space equipped with an inner product $\scalar{\cdot}{\cdot}$, and let $\|\cdot\|$ be a gauge on $V$ with dual gauge  $\|\cdot\|^\ast$. If $x\in V$, then 
    \begin{align*}
  \pi \in \partial\|x\|\qquad\iff\qquad     \|x\|=\scalar{\pi}{x} \quad
	\mbox{and}\quad \|\pi\|^\ast\leq 1.
    \end{align*}
\end{lemma}

\begin{remark}
If $x\neq0$, then the condition $\|x\|=\scalar{\pi}{x}$ automatically implies $\|\pi\|^\ast\geq 1$. Indeed, by definition of the dual gauge,
  \[
  \|\pi\|^\ast = \sup_{\substack{y\in V\\\|y\|\leq 1}}  \scalar{\pi}{y} \geq \langle \,\pi\, |\, \frac{x}{\|x\|}\,\rangle=1.  
  \]
  Thus, when $x\neq0$, the characterization of the gauge subdifferential can be equivalently stated as
  \begin{align*}
  \pi \in \partial\|x\|\qquad\iff\qquad     \|x\|=\scalar{\pi}{x} \quad
	\mbox{and}\quad \|\pi\|^\ast= 1.
    \end{align*}
\end{remark}

For a nonzero linear subspace $\SS$ of $V$, define the restricted dual gauge on $\SS$ by
\[
\|y\|_\SS^\ast = \sup_{\substack{x\in S\\\|x\|\leq 1}}  \scalar{y}{x},\qquad y\in \SS. 
\]
It is clear that
$\| y\|^\ast \geq \|y\|_\SS^\ast$ for $y\in \SS$,
because the supremum in $\|y\|_\SS^\ast$ is being taken over a smaller set. 

\begin{lemma}
	For $y\in\SS$, we have
	\[
	\|y\|_\SS^\ast = \min_{\xi\in \SS^\bot} \|y+\xi\|^\ast. 
	\]
\end{lemma}
\begin{proof}
	For any $\xi\in\SS^\bot$ and $x, y\in\SS$, we have 
	\[
	\scalar{y}{x} = \scalar{y+\xi}{x}\leq \|y+\xi\|^\ast \|x\|.
	\]
	This implies that for any $\xi\in \SS^\bot$,
	\[
	\|y\|_\SS^\ast = \sup_{\substack{x\in \SS\\\|x\|\leq 1}}  \scalar{y}{x} \leq \|y+\xi\|^\ast \sup_{\substack{x\in \SS\\\|x\|\leq 1}} \|x\|  = \|y+\xi\|^\ast.
	\]
	Fix $y\in\SS$. For $\xi\in V$ define 
	\[
	f(\xi)=\|y+\xi\|^\ast\quad \mbox{ and }\quad\delta_{\SS^\bot}(\xi) = \begin{cases} 0, & \xi \in \SS^\bot, \\ \infty & \mbox{otherwise}.\end{cases}
	\]
	Clearly 
	\[
	\min_{\xi\in\SS^\bot} f(\xi) = \min_{\xi\in V} \left\{ f(\xi)+\delta_{\SS^\bot}(\xi)\right\}.
	\]
	Because $f$ is continuous, convex and coercive on a closed subspace $\SS^\bot$, the minimum of $f$ is attained. 
	At an optimal point $\xi^\ast$ we have
	$0\in \partial f(\xi^\ast)+\partial\delta_{\SS^\bot}(\xi^\ast)$. Therefore, there exist $x^\ast\in \partial f(\xi^\ast)$ and $w^\ast\in \partial\delta_{\SS^\bot}(\xi^\ast)$ with $x^\ast+w^\ast=0$.  It is well known that $\partial\delta_{\SS^\bot}(\xi^\ast)=\SS$, which implies that $x^\ast=-w^\ast\in\SS$. Since $x^\ast\in \partial f(\xi^\ast)=\partial \|y+\xi^\ast\|^\ast$, by Lemma \ref{lem:subd}, we obtain that $\|x^\ast\|\leq 1$ and 
	$\scalar{y+\xi^\ast}{x^\ast} = \|y+\xi^\ast\|^\ast$.    Thus, 
	\[
	\|y\|_\SS^\ast = \sup_{\substack{x\in \SS\\\|x\|\leq 1}}  \scalar{y}{x} \geq \scalar{y}{x^\ast} = \scalar{y+\xi^\ast}{x^\ast}=\min_{\xi\in\SS^\bot}\|y+\xi\|^\ast,
	\]
	which ends the proof.
\end{proof}

Let $\opnorm{\cdot}$ denote the operator gauge induced by the gauge $\|\cdot\|$, that is, for a linear operator $A$ on $V$, 
\[
\opnorm{A} = \sup_{\substack{x\in V\\\|x\|\leq 1}}  \| A x\|. 
\]
Also, let $\PP_\SS$ be the orthogonal projection from $V$ onto $\SS$.
\begin{theorem}\label{thm:restricted}
	We have 
	\[
	\| y\|^\ast = \|y\|_\SS^\ast,\qquad \forall\, y\in \SS,
	\]
	if and only if $\opnorm{\PP_\SS}= 1$.
\end{theorem}
\begin{proof}
	Assume that $\opnorm{\PP_\SS}= 1$.  Then, for any $x\in V$, we have $z = \PP_{\SS}x\in\SS$ and
	\[
	\|z\|\leq  \opnorm{\PP_{\SS}}   \| x\|=  \| x\|.
	\]
	Then, for $y\in\SS$, by definition of $\PP_{\SS}$, 
	\[
	\| y\|^\ast  = \sup_{\substack{x\in V\\\|x\|\leq 1}}  \scalar{y}{x} = \sup_{\substack{x\in V\\\|x\|\leq 1}}  \scalar{\PP_{\SS} y}{x} = \sup_{\substack{x\in V\\\|x\|\leq 1}}  \scalar{ y}{\PP_{\SS} x} \leq \sup_{\substack{z\in \SS\\\|z\|\leq 1}}  \scalar{ y}{z} = \|y\|_\SS^\ast.
	\]
	Since we already know that $\| y\|^\ast \geq \|y\|_\SS^\ast$, we obtain $\| y\|^\ast = \|y\|_\SS^\ast$. 
	
	Conversely, assume that  $\| y\|^\ast = \|y\|_S^\ast$ for any $y\in\SS$. 
	We have 
	\begin{align*}
		\|\PP_\SS x\| &= \sup_{\substack{y\in \SS\\\|y\|_\SS^\ast\leq 1}}\scalar{\PP_\SS x}{y} = \sup_{\substack{y\in \SS\\\|y\|_\SS^\ast\leq 1}}\scalar{x}{y} = \sup_{\substack{y\in \SS\\\|y\|^\ast\leq 1}}\scalar{x}{y} \\
		& \leq \sup_{\substack{y\in V\\\|y\|^\ast\leq 1}}\scalar{x}{y}  = \|x\|.
	\end{align*}
	Thus, $\opnorm{\PP_\SS}\leq 1$. Since $\PP_{\SS}^2=\PP_{\SS}$, we obtain $\opnorm{\PP_\SS} \leq \opnorm{\PP_\SS}^2$, i.e., $\opnorm{\PP_\SS}\geq 1$.
\end{proof}

\end{document}